\documentclass[12pt, a4paper, reqno, oneside]{amsart}
\usepackage{amssymb}  % for \leqslant, \geqslant
\usepackage{amsmath} % for \pmod with smaller space before (mod ...)
\usepackage{amsthm}  % for \begin{proof} \end{proof}   !!! defines \openbox !!!

\usepackage[cp1251]{inputenc}
\usepackage[english]{babel}
\usepackage{cite, verbatim}% дополнительные пакеты

\newtheorem{theorem}{Theorem}
\newtheorem{conj}{Conjecture}
\newtheorem{lemma}{Lemma}
\numberwithin{lemma}{section}

\numberwithin{equation}{section}

\newcommand{\Aut}{\operatorname{Aut}}

\newcommand{\Inndiag}{\operatorname{Inndiag}}

\newcommand{\gexp}{\operatorname{exp}}

\begin{document}
%\noindent УДК 512.542

\vspace{1cm}

\title[On recognition of symplectic and orthogonal groups]{On recognition of symplectic and orthogonal groups of small dimensions by spectrum}

\author{M.A. Grechkoseeva}

\author{A.V. Vasil'ev}

\author{M.A. Zvezdina}

\thanks{The work is supported by  Russian Science Foundation (project 14-21-00065).}

\begin{abstract}
We refer to the set of the orders of elements of a finite group as its spectrum and say that finite groups are isospectral if their spectra coincide.
In the paper we determine all finite groups isospectral to the simple groups $S_6(q)$, $O_7(q)$, and $O_8^+(q)$.
In particular, we prove that with just four exceptions, every such a finite group is an extension of the initial simple group by a (possibly trivial) field automorphism.

 \end{abstract}

\maketitle

\section{Introduction}

The spectrum  $\omega(G)$ of a finite group $G$ is the set of the orders of its elements. Groups whose spectra coincide are said to be isospectral.
If $G$ is a finite group having a nontrivial normal solvable subgroup, then by \cite[Lemma 1]{98Maz.t}, there are infinitely many pairwise nonisomorphic finite groups isospectral to $G$.
In contrast, the finite nonabelian simple groups are rather satisfactorily determined by the spectrum. We refer to a nonabelian simple group $L$ as recognizable by spectrum if
every finite group $G$ isospectral to $L$ is isomorphic to $L$, and as almost recognizable by spectrum if every such a group $G$ is an almost simple group with socle isomorphic to $L$.
It is known that all sporadic and alternating groups, except for $J_2$, $A_6$ and $A_{10}$, are recognizable by spectrum (see \cite{98MazShi, 13Gor.t}) and all exceptional groups excluding $^3D_4(2)$ are
almost recognizable by spectrum (see \cite{14VasSt.t,16Zve.t}). In 2007 V.D. Mazurov conjectured that there is a positive integer $n_0$ such that all simple classical groups of dimension at least $n_0$
are as well almost recognizable by spectrum. Mazurov's conjecture was proved in \cite[Theorem 1.1]{15VasGr1} with $n_0=62$. Later it was shown in \cite[Theorem 1.2]{17Sta} that we can take $n_0=38$.
It is clear that this bound is far from being final, and we conjecture that the following holds \cite[Conjecture 1]{15VasGr1}.

\begin{conj} \label{c:1} Suppose that $L$ is one of the following nonabelian simple groups:
\begin{enumerate}
\item $L_n(q)$, where $n\geqslant5$;
\item $U_n(q)$, where $n\geqslant5$ and $(n,q)\neq(5,2)$;
\item $S_{2n}(q)$, where $n\geqslant3$, $n\neq4$ and $(n,q)\neq(3,2)$;
\item $O_{2n+1}(q)$, where $q$ is odd, $n\geqslant3$, $n\neq4$ and $(n,q)\neq(3,3)$;
\item $O_{2n}^\varepsilon(q)$, where $n\geqslant4$ and $(n,q,\varepsilon)\neq (4,2,+),(4,3,+)$.
\end{enumerate}
Then every finite group isospectral to $L$ is an almost simple group with socle isomorphic to $L$.
\end{conj}

In the present paper we prove the part of Conjecture \ref{c:1} concerning the groups $S_6(q)$, $O_7(q)$ and $O_8^+(q)$. More precisely, we determine finite groups isospectral to these simple groups.
%This description involves field automorphisms of the groups under consideration.
%If $L$ is one of the groups $S_6(q)$, $O_7(q)$, or  $O_8^+(q)$, where $q=p^m$ and $p$ is a prime, then $L$ has a field automorphism $\varphi$
%of order $m$

%% !полевые автоморфизмы

\begin{theorem} \label{t:1} Let $L$ be one of the simple groups $S_6(q)$, $O_7(q)$, or $O_8^+(q)$, where $q$ is a power of a prime $p$, and let $G$ be a finite group.
\begin{enumerate}
\item If $q$ is odd, $p\neq 5$ and $L\neq O_7(3), O_8^+(3)$, then $\omega(G)=\omega(L)$ if and only if up to isomorphism
$G=L\rtimes\langle \varphi\rangle$, where $\varphi$ is a (possibly trivial) field automorphism of $L$ whose  order is a power of $2$.
\item If $q>2$ is even or $p=5$, then $\omega(G)=\omega(L)$ if and only if $G\simeq L$.
\item If $L=O_7(3)$ or $L=O_8^+(3)$, then $\omega(G)=\omega(L)$ if and only if $G\simeq O_7(3)$ or $G\simeq O_8^+(3)$.
\item If $L=S_6(2)$ or $L=O_8^+(2)$, then $\omega(G)=\omega(L)$ if and only if $G\simeq S_6(2)$ or $G\simeq O_8^+(2)$.
\end{enumerate}
\end{theorem}

Theorem \ref{t:1} has been proved for $q=3$ and  $q$ even (see  \cite{97ShiTan} and \cite{12Sta.t} respectively), so we consider the case when $q$ is odd and $q>3$.
In this case the structure of groups isospectral to $L$ is also well studied and to prove Theorem \ref{t:1}, it remains to establish the following.

\begin{theorem} \label{t:2} Let $L$ be one of the simple groups $S_6(q)$, $O_7(q)$,  or $O_8^+(q)$, where $q$ is a power of an odd prime $p$ and $q>3$.
If $G$ is a finite group and $\omega(G)=\omega(L)$, then any simple group of Lie type in characteristic other than $p$ cannot be a composition factor of $G$.
\end{theorem}

Note that the situation with $S_6(q)$, $O_7(q)$, and $O_8^+(q)$ is typical for all simple groups mentioned in Conjecture \ref{c:1} but not covered by \cite{15VasGr1, 17Sta}: the conclusion of the conjecture
is proved either for some specific $n$ and $q$, mostly when the prime graph of $L$ is  disconnected, or for even $q$, while in other cases,  it remains to prove an assertion similar to Theorem~\ref{t:2}.
There are few works that solve the described problem of eliminating groups of Lie type in cross characteristic when the prime graph of $L$ is connected and  the characteristic of $L$ is odd,
and even fewer general methods of solution. It is one of our aims to find such methods, and in particular Lemma \ref{l:a} can be regarded as a step in this direction.

\section{Preliminaries}

As usual, $[m_1,m_2,\dots,m_k]$ and $(m_1,m_2,\dots,m_k)$ denote respectively the least common multiple and greatest common divisor of integers $m_1,m_2,\dots,m_s$.
Given a positive integer $m$, we write $\pi(m)$ for the set of prime divisors of $m$. Also for a prime $r$, we write $(m)_{r}$ for the $r$-part of $m$, that is, the highest power of $r$ dividing $m$,
and $(m)_{r'}$ for the $r'$-part of $m$, that is, the ratio $m/(m)_{r}$. If $\varepsilon\in\{+,-\}$, then in arithmetic expressions, we abbreviate $\varepsilon 1$ to $\varepsilon$. The integer value of a real number $x$ is denoted by $[x]$.

The next two lemmas  are well known (see, for example, \cite[Ch. IX, Lemma 8.1]{82HupBl2}). We write $\Phi_m(x)$ to denote the $m$th cyclotomic polynomial.

\begin{lemma}\label{l:r-part_c}
Let $r$ be a prime, let $a$ be an integer prime to $r$ with $|a|>1$ and let $k$ be the multiplicative order of $a$  modulo $r$.
If $r$ is odd, then 
\begin{enumerate}
\item $(\Phi_k(a))_r>1$;
\item $(\Phi_{kr^i}(a))_r=r$ for all $i\geqslant 1$;
\item $(\Phi_{m}(a))_r=1$ for all other $m\geqslant 1$.
\end{enumerate}
Also if $r=2$, then 
\begin{enumerate}
\item $(\Phi_{r^i}(a))_r=r$ for all $i\geqslant 2$;
\item $(\Phi_{m}(a))_r=1$ for all  $m\neq r^i$ $(i\geqslant 0)$.
\end{enumerate}

\end{lemma}

\begin{lemma}\label{l:r-part}
Let $a$ and $m$ be positive integers and let $a>1$. Suppose that $r$ is a prime and  $a\equiv \varepsilon\pmod r$, where  $\varepsilon\in\{+1,-1\}$.
\begin{enumerate}
\item If $r$ is odd, then $(a^m-\varepsilon^m)_{r}=(m)_r(a-\varepsilon)_{r}$.
\item If $a\equiv 1\pmod 4$, then $(a^m-1)_{2}=(m)_{2}(a-1)_{2}$.
\end{enumerate}
\end{lemma}

Let $a$ be an integer. If $r$ is an odd prime and $(a,r)=1$, then $e(r,a)$ denotes the multiplicative order of
$a$ modulo $r$. Define $e(2,a)$ to be 1 if  $4$ divides $a-1$ and to be $2$ if $4$ divides $a+1$.
A primitive prime divisor of $a^m-1$, where $|a|>1$ and $m\geqslant 1$, is a prime $r$ such that $e(r,a)=m$.
The set of primitive prime divisors of $a^m-1$ is denoted by $R_m(a)$,
and we write $r_m(a)$ for an element of $R_m(a)$ (provided that it is not empty). The following lemma is proved in \cite{86Bang}, and also in \cite{Zs}.

\begin{lemma}[Bang--Zsigmondy]\label{l:bz}
Let $a$ and $m$ be integers, $|a|>1$ and $m\geqslant 1$. Then the set $R_m(a)$ is not empty, except when $$(a,m)\in\{(2,1),(2,6),(-2,2),(-2,3),(3,1),(-3,2)\}.$$
\end{lemma}

The largest primitive divisor of $a^m-1$, where $|a|>1$, $m\geqslant 1$, is the number $k_m(a)=\prod_{r\in R_m(a)}|a^m-1|_{r}$ if $m\neq 2$ and
the number $k_2(a)=\prod_{r\in R_2(a)}|a+1|_{r}$ if $m=2$. The largest primitive divisors can be written in terms of cyclotomic polynomials. 

\begin{lemma}\label{l:k_n}
Let $a$ and $m$ be integers, $|a|>1$ and $m\geqslant 3$. Suppose that $r$ is the largest prime divisor of $m$ and $l=(m)_{r'}$. Then
$$k_m(a)=\frac{|\Phi_m(a)|}{(r,\Phi_l(a))}.$$
Furthermore, $(r,\Phi_l(a))=1$ whenever $l$ does not divide $r-1$.
\end{lemma}

\begin{proof}
This follows from \cite[Proposition 2]{97Roi} (see, for example, \cite[Lemma 2.2]{15VasGr.t}).
\end{proof}

Recall that $\omega(G)$ is the set of the orders of elements of $G$. We write $\mu(G)$ for the set of maximal under divisibility elements of $\omega(G)$. The least common multiple of the elements of $\omega(G)$ is equal to
the exponent of $G$ and denoted by $\gexp(G)$. Given a prime $r$, $\omega_{r}(G)$ and $\gexp_r(G)$ are respectively the spectrum and the exponent of a Sylow $r$-subgroup of $G$. Similarly, $\omega_{r'}(G)$ and $\gexp_{r'}(G)$ are respectively the set of the orders of elements of $G$ that are coprime to $r$ and the least common multiple of these orders.

The prime graph $GK(G)$ of $G$ is a labelled graph whose vertex set is $\pi(G)$, the set of all prime divisors of $|G|$, and in which two different vertices labelled by $r$ and $s$ are adjacent if
and only if $rs\in\omega(G)$.
Recall that a coclique of a graph is the set of pairwise nonadjacent vertices. Define $t(G)$ to be the largest size of a coclique of $GK(G)$. Similarly, given $r\in\pi(G)$, we write
$t(r,G)$ for the largest size of a coclique of $G$ containing $r$. It was proved in \cite{05Vas.t} that a finite group $G$ with $t(G)\geqslant 3$ and $t(2,G)\geqslant 2$ has exactly one nonabelian composition factor.
The next lemma is a corollary of this result.

\begin{lemma}[{\cite{05Vas.t, 09VasGor.t}}]\label{l:str}
Let $L$ be a finite nonabelian simple group such that $t(L)\geqslant3$ and $t(2,L)\geqslant2$, and suppose that a finite group $G$
satisfies $\omega(G)=\omega(L)$. Then the following holds.
\begin{enumerate}
\item There is a nonabelian simple group $S$ such that $$S\leqslant \overline G=G/K\leqslant \Aut S,$$ with $K$ being
the largest normal solvable subgroup of $G$.

\item If $\rho$ is a coclique of $GK(G)$ of size at least  $3$, then at most one prime of $\rho$
divides $|K|\cdot|\overline{G}/S|$. In particular, $t(S)\geqslant t(G)-1$.

\item If $r\in\pi(G)$ is not adjacent to $2$ in $GK(G)$, then $r$ does not divide
$|K|\cdot|\overline{G}/S|$. In particular,  $t(2,S)\geqslant t(2,G)$.
\end{enumerate}
\end{lemma}

The next three lemmas are standard tools for calculating the orders of elements in group extensions. 
All of them are corollaries of well-known results (such as the Hall--Higman theorem).

\begin{lemma}\label{l:nonc} Suppose that $G$ is a finite group,  $K$ is a normal subgroup of $G$
and $r\in\pi(K)$. If $G/K$ has a section that is a noncyclic abelian $p$-group
for some odd prime $p\neq r$, then $rp\in\omega(G)$.
\end{lemma}

\begin{proof}
See \cite[Lemma 1.5]{11VasGrSt.t}.
\end{proof}

\begin{lemma} \label{l:frob} Suppose that $G$ is a finite group,  $K$ is a normal $r$-subgroup of $G$
for some prime $r$ and $G/K$ is a Frobenius group with kernel $F$ and cyclic complement $C$. If
$(|F|, r) = 1$ and $F$ is not contained in $KC_G(K)/K$, then $r|C| \in\omega(G)$.
\end{lemma}

\begin{proof}
See \cite[Lemma 1]{97Maz.t}.
\end{proof}

\begin{lemma}\label{l:hh} Let $p$ and $s$ be primes such that $p\neq s$ and let $G$ be a semidirect product of a finite $p$-group $T$ and a cyclic group $\langle g\rangle$ of order $s$.
Suppose that $[T,g]\neq 1$ and $G$ acts faithfully on a vector space $V$ of positive characteristic  $r\neq p$. Then either the natural semidirect product $V\rtimes G$ has an element of order $sr$,
or the following holds:
\begin{enumerate}
 \item $C_T(g)\neq 1$;
 \item $T$ is nonabelian;
 \item $p=2$ and $s$ is a Fermat prime.
\end{enumerate}
\end{lemma}

\begin{proof}
See \cite[Lemma 3.6]{15Vas}.
\end{proof}

If $G$ is a group of the hypothesis of Lemma \ref{l:str}, then imposing some additional restrictions on $\omega(G)$, we can
guarantee that the solvable radical of $G$ is nilpotent.

\begin{lemma}\label{l:nilp}
Let $G$ be a finite group and let $S\leqslant G/K\leqslant \Aut S$, where $K$ is a normal solvable subgroup of $G$ and $S$ is a nonabelian simple group.
Suppose that for every $r\in\pi(K)$, there is $a\in\omega(S)$ such that $\pi(a)\cap\pi(K)=\varnothing$
and $ar\not\in\omega(G)$. Then $K$ is nilpotent.
\end{lemma}

\begin{proof}
Otherwise, the Fitting subgroup $F$ of $K$ is a proper subgroup of $K$.
Define $\widetilde G=G/F$ and $\widetilde K=K/F$. Let $\widetilde T$ be a minimal normal subgroup of
$\widetilde G$ contained in $\widetilde K$ and let $T$ be its preimage in $G$. It is clear that $\widetilde T$ is an elementary abelian $t$-group
for some prime $t$. Given $r\in \pi(F)\setminus \{t\}$, denote the Sylow $r$-subgroup of $F$ by $R$,
its centralizer in $G$ by $C_r$ and the image of $C_r$ in $\widetilde G$ by $\widetilde C_r$. Since $\widetilde C_r$ is normal in $\widetilde G$,
it follows that either $\widetilde T\leqslant \widetilde C_r $ or $\widetilde C_r \cap \widetilde T = 1$.
If $\widetilde T\leqslant \widetilde C_r$ for all $r\in\pi(F)\setminus \{t\}$, then $T$ is a normal nilpotent subgroup of $K$, which contradict the choice of $\widetilde T$.
Thus there is $r\in\pi(F)\setminus\{t\}$ such that $\widetilde C_r \cap \widetilde T = 1$.

If $C_{\widetilde G}(\widetilde T)$ is not contained in $\widetilde K$, then it has a section isomorphic to $S$. In this case $ta\in\omega(G)$ for every $a\in\omega_{t'}(S)$, contrary
to the hypothesis. Thus  $C_{\widetilde G}(\widetilde T)\leqslant \widetilde K$.

Choose $a\in\omega_{r'}(S)$ such that $\pi(a)\cap\pi(K)=\varnothing$ and $ra\not\in\omega(G)$, and let $x\in \widetilde G$ be an element of order $a$. Then
$x\not\in C_{\widetilde G}(\widetilde T)$, therefore, $[\widetilde T,x]\neq 1$ and so $[\widetilde T,x]\rtimes\langle x\rangle$ is a Frobenius group with complement $\langle x\rangle$.
Since $\widetilde C_r \cap \widetilde T = 1$, we can apply Lemma \ref{l:frob} and conclude that $ra\in\omega(G)$, contrary to the choice of $a$.
\end{proof}

The last two lemmas are concerned with the numbers $k_3(q)$, $k_4(q)$ and $k_6(q)$, which are related to the spectra of $S_6(q)$, $O_7(q)$ and $O_8^+(q)$. Applying Lemma \ref{l:k_n},
we see that $$k_3(q)=\frac{q^2+q+1}{(3,q-1)}, \quad k_4(q)=\frac{q^2+1}{(2,q-1)}, \quad k_6(q)=\frac{q^2-q+1}{(3,q+1)}.$$

\begin{lemma}\label{l:nl4}
Let $q$ be a prime power and let $r$ be a prime. If $q^2+1=2r^l$, with
$l>1$, then either $l=2$ and $q$ is a prime, or $l=4$, $q=239$ and $r=13$.
\end{lemma}

\begin{proof}
See Lemmas 3 and 4 in \cite{75Cre} (Lemma 1 of \cite{75Cre} misses the case $q=3$).
\end{proof}

\begin{lemma}\label{l:k3}
Let $q$ be an odd prime power, $q\geqslant 7$ and $q\equiv \varepsilon \pmod 4$, where $\varepsilon \in\{+1,-1\}$.
\begin{enumerate}
 \item Every prime divisor of $k_3(\varepsilon q)$ or $k_6(\varepsilon q)$ is congruent to $1$ modulo $3$.
 \item $k_6(\varepsilon q)>16q^2/51$, $k_3(\varepsilon q)>10q^2/33$ and  $k_3(\varepsilon q), k_6(\varepsilon q)\geqslant 19$.
 \item If $k_3(\varepsilon q)\leqslant 241$, then $q$ is contained in the following table:
 $$\begin{array}{|c|c|c|c|c|c|c|}
\hline
q&7&9&11&13&23&25\\
\hline
k_3(\varepsilon q)&43&7\cdot 13&37& 61&13^2&7\cdot 31\\
\hline
\end{array}$$
 \item If $k_6(\varepsilon q)\leqslant 757$, then $q$ is contained in the following table:
$$\begin{array}{|c|c|c|c|c|c|c|c|c|c|c|c|c|c|}
\hline
q&7&9&11&13&17&19&23&25&27&29&31&41&43\\
\hline
k_6(\varepsilon q)& 19&73&7\cdot 19& 157&7\cdot 13&127&7\cdot 79&757&601&271&331&547&631\\
\hline
\end{array}$$
\item The number $k_6(\varepsilon q)$ cannot be equal to either of the numbers $k_8(7)=1201$, $k_7(3)=1093$, $k_7(4)=43\cdot 127$, $k_7(5)=19531$, $k_7(8)=127\cdot 337$, $k_7(9)=547\cdot 1093$, $k_7(17)=41761$, $13\cdot 61$, $1321$.

\end{enumerate}
\end{lemma}
\begin{proof}

 (i) This follows from the definition of a primitive divisor and Little Fermat's Theorem.

 (ii) It is clear that $k_3(q), k_6(q)>(q^2-q)/3$. If $q\geqslant 17$, then $$(q^2-q)/3=q^2(q-1)/3q\geqslant 16q^2/51>10q^2/33>19.$$ For $q=7, 9,11,13$, the desired inequalities
 are verified by direct calculation.

 (iii)-(iv) If $q\geqslant 47$, then $k_6(\varepsilon q)> 601$ and $k_3(\varepsilon q)>241$ by (ii). For $q\leqslant 43$, the assertion follows by direct calculation.

(v) Suppose that $k_6(\varepsilon q)=k_8(7)=1201$. Then $q(q-\varepsilon)$ is equal to either $1200=16\cdot 3\cdot 25$ or $3\cdot 1201-1=2\cdot 1801$.
Since $q$ is an odd prime power, $q$ is equal to the largest odd primary divisor of $q(q-\varepsilon)$. Thus $q=25$ and $q-\varepsilon=48$, or $q=1081$ and $q-\varepsilon=2$, and this is a contradiction.
The other cases are handled in a similar manner.

\end{proof}

\section{Spectra and exponents of groups of Lie type}

In this section we give some lower bounds on the exponents of simple groups of Lie type and list the spectra of some groups of low Lie rank.
Throughout the paper we repeatedly use, mostly without explicit references,  the description of the spectra of simple classical groups from \cite{10But.t} (with corrections from \cite[Lemma 2.3]{16Gr.t}) and \cite{08But.t},
as well as the adjacency criterion for the prime graphs of simple groups of Lie type from \cite{05VasVd.t} (with corrections from \cite{11VasVd.t}).
Also we use the abbreviations $L_n^\tau(u)$ and $E_6^\tau(u)$, where $\tau\in\{+,-\}$, that are defined as follows:  $L_n^+(u)=L_n(u)$, $L_n^-(u)=U_n(u)$, $E_6^+(u)=E_6(u)$ and $E_6^-(u)={}^2E_6(u)$.

\begin{lemma}\label{l:spec_c3}
Let $q$ be a power of an odd prime $p$. Let $L=S_6(q)$ and $d=1$, or $L=O_7(q)$ and $d=2$. Then $\omega(L)$ consists of the divisors of the following numbers:
\begin{enumerate}
  \item $(q^3\pm 1)/2$, $(q^2+1)(q+1)/2$, $(q^2+1)(q-1)/2$, $q^2-1$, $p(q^2\pm1)/d$;
  \item $9(q\pm1)/d$ if $p=3$;
  \item $25$ if $p=5$.
\end{enumerate}
The set $\omega(O_8^+(q))$ is the union of $\omega(O_7(q))$ and the set of divisors of $(q^4-1)/4$.
\end{lemma}

\begin{lemma} \label{l:expc3} Let $q$ be a power of an odd prime $p$ and let $L$ be one of the groups $S_6(q)$, $O_7(q)$, or $O_8^+(q)$. Then
$$\gexp(L)=\begin{cases} p^2(q^6-1)(q^2+1)/2&\text{ if } p=3,5\\
                         p(q^6-1)(q^2+1)/2&\text{ if } p>5.\\
           \end{cases}
$$
If $q\geqslant 7$, $a=p(q^2+1)/2$ and $b=(q^3-1)/2$, then $\gexp(L)$ is less than any of the numbers
$q^9$, $6q^6b/5$, $5b^3$, and $a^4$.
\end{lemma}

\begin{proof}
The first assertion follows from Lemma \ref{l:spec_c3}. If $q\geqslant 7$, then $\gexp_p(L)\leqslant q$ and so
$\gexp(L)<q^9$. Furthermore,
$$\gexp(L)/b\leqslant q(q^2+1)(q^3+1)<6q^6/5<5(q^3-1)^2/4= 5b^2,$$
$$\gexp(L)/a\leqslant p(q^6-1)<\frac{p^3q^6}{8}<a^3.$$
\end{proof}

We will write the exponent of a classical group of Lie rank $n$ over a field of order $q$ in terms of the product $\prod_{i=1}^{n}|\Phi_i(q^k)|$, where $k$ depends on the type of the group,
so we will need some lower bounds on the number $|\Phi_i(a)|$ and on the function $F(n)$, where
$$%\label{e:sum_varphi}
 F(n)=\sum_{i=1}^{n}\varphi(i).
$$

\begin{lemma}\label{l:cycl}
Let $\varepsilon\in\{+1,-1\}$ and let $a$ and $n$ be positive integers. If $a\geqslant 2$ and $n\geqslant 3$, then $|\Phi_n(\varepsilon a)|>a^{3\varphi(n)/4}$.
In particular, if $a\geqslant 2$ and $n\geqslant 2$, then $\prod_{i=1}^{n}|\Phi_i(\varepsilon a)|>a^{3F(n)/4}$.
\end{lemma}

\begin{proof}
Because of well-known relations between cyclotomic polynomials, it suffices to consider the case $\varepsilon =+1$ only.

Given a a prime divisor $r$ of $n$, set $m=(n)_{r'}$ and $(n)_r=r^k$. Applying \cite[Lemma 1]{97Roi} and the condition $a\geqslant 2$, we have $\Phi_n(a)>a^{r^{k-1}(r-2)\varphi(m)}$.
If $r\geqslant 5$, then  $(r-2)/(r-1)\geqslant 3/4$, and hence $r^{k-1}(r-2)\varphi(m)\geqslant 3\varphi(n)/4$.

Thus we can assume that $n=2^k3^j$. If $n$ is divisible by $r^2$, with $r=2,3$, then $n/r\geqslant r$,
$\Phi_n(a)=\Phi_{n/r}(a^r)$ and $\varphi(n)=r\varphi(n/r)$, so working by induction on $n$, we can assume that $n$ is equal to $3$, or $4$, or $6$.
In this situation $\Phi_n(a)\geqslant a^2-a+1>a^{3/2}$, as desired.

The last assertion follows from the above and the inequality $\Phi_1(a)\Phi_2(a)=a^2-1>a^{3/2}$.

\end{proof}

\begin{lemma}\label{l:sum_varphi}
If $n\geqslant 1$, then $F(n)\geqslant [(n+1)/2]^2$.
\end{lemma}

\begin{proof}
 %http://pieterjandesmetcalculations.blogspot.tu/2015/05/upper-and-lower-bounds-on-totient.html
This proof is due to \cite{PJDS}. Extend the definition of $F(n)$ by setting $F(x)=\sum_{i\leqslant x}\varphi(i)$.
Then by \cite[Theorem 3.11]{76Apo}, it follows that
$$\sum_{i\leqslant x}F(x/i)=\sum_{i\leqslant x}\sum_{d\mid i}\varphi(d)=\sum_{i\leqslant x}i.$$ Writing $I(x)=\sum_{i\leqslant x}i$, we have
$F(n)+F(n/2)+\dots+F(1)=I(n)$, and hence $$I(n)-2I(n/2)=F(n)-F(n/2)+F(n/3)-\dots\leqslant F(n).$$
If $n$ is odd, then $$I(n)-2I(n/2)=(n+1)n/2-(n+1)(n-1)/4=(n+1)^2/4,$$ while for even $n$, we have $I(n)-2I(n/2)=n^2/4$.
\end{proof}

\begin{lemma}\label{l:exp} Let $u$ be a power of a prime $v$.
\begin{enumerate}

\item Let $S=L_n^\tau(u)$, where $n\geqslant 3$. Then $\gexp_{v'}(S)=\prod_{i=1}^{n} |\Phi_i(\tau u)|/c$, where $c=r\in\pi(q-\tau)$ if $n=r^s$ and $c=1$ otherwise.
In particular, $$\gexp(S)\geqslant \frac{n}{c}\cdot \prod_{i=1}^{n} |\Phi_i(\tau u)|>\frac{n}{c}\cdot u^{3F(n)/4}\geqslant u^{3F(n)/4}.$$

\item Let $S=S_{2n}(u)$ or $S=O_{2n+1}(u)$, where $n\geqslant 2$. Then $\gexp_{v'}(S)=\prod_{i=1}^{n} \Phi_{i}(u^2)/c$, where $c=(2,u-1)^2$ if $n=2^s$ and $c=(2,u-1)$ otherwise.
In particular, $$\gexp(S)>\frac{2n}{c}\cdot \prod_{i=1}^{n} \Phi_i(u^2)>\frac{n}{2}\cdot u^{3F(n)/2}.$$

\item Let $S=O_{2n}^-(u)$, where $n\geqslant 4$. If $n$ is even, then $\gexp_{v'}(S)=\gexp_{v'}(O_{2n+1}(q))$ and $$\gexp(S)\geqslant \frac{2n-1}{c}\cdot \prod_{i=1}^{n} \Phi_i(u^2)>\frac{2n-1}{4}\cdot u^{3F(n)/2},$$
where $c$ is as in \rm{(ii)}. If $n$ is odd, then $\gexp_{v'}(S)=\Phi_{2n}(u)\prod_{i=1}^{n-1} \Phi_{i}(u^2)/(2,u-1)$ and $$\gexp(S)\geqslant \frac{2n-1}{(2,u-1)}\cdot \Phi_{2n}(u)\prod_{i=1}^{n-1} \Phi_i(u^2)> \frac{2n-1}{2}\cdot u^{3(F(n)+F(n-1))/4}.$$

\item Let $S=O_{2n}^+(u)$, where $n\geqslant 4$. If $n$ is even, then $\gexp(S)=\gexp(O_{2n-1}(q))$. If $n$ is odd, then $\gexp_{v'}(S)=\Phi_{n}(u)\prod_{i=1}^{n-1} \Phi_{i}(u^2)/(2,u-1)$ and $$\gexp(S)\geqslant \frac{2n-1}{(2,u-1)}\cdot \Phi_{n}(u)\prod_{i=1}^{n-1} \Phi_i(u^2)>\frac{2n-1}{2}\cdot u^{3(F(n)+F(n-1))/4}.$$
\end{enumerate}
\end{lemma}

\begin{proof} All bounds on $\gexp_{v}(S)$ below follow from \cite[Proposition 0.5]{95Tes}.

(i) Let $A=\prod_{i=1}^{n}|\Phi_i(\tau u)|$ and fix $r\in\pi(S)\setminus\{v\}$. Set $j=e(r,\tau u)$ if $r$ is odd and $j=1$ if $r=2$, and choose the largest $s$ such that $jr^s\leqslant n$.
Then $$(A)_r=(u^{jr^s}-\tau^{jr^s})_r.$$ Indeed, if $i\leqslant n$ and $i$ does not divide  $jr^s$, then $r$ does not divide $\Phi_i(\tau u)$ by Lemma \ref{l:r-part_c}, and the product of $\Phi_i(\tau u)$ over all $i$ dividing   $jr^s$ is equal to $(\tau u)^{jr^s}-1$.

If $r$ is coprime to $u-\tau$ or $jr^s\leqslant n-2$, then $\gexp_r(S)=(u^{jr^s}-\tau^{jr^s})_r$. Suppose that $r$ divides $u-\tau$, in other words,  $j=1$, and also $r^s\geqslant n-1$. The last inequality yields $s>0$. If $r^s=n-1$, then $(r,n)=1$, so in this case $\gexp_r(S)=(u^{jr^s}-\tau^{jr^s})_r$ as well. If $r^s=n$, then $n/r<n-1$ and $\gexp_r(S)=(u^{n/r}-\tau^{n/r})_r=(A)_r/r$, where the last equality holds by Lemma \ref{l:r-part} (observe that $n/r$ is even if $r=2$). Thus we have the desired formula for $\gexp_{v'}(S)$. Combining this with the inequality $\gexp_v(S)\geqslant n$, we derive the first  bound for $\gexp(S)$.
The further bounds follow from Lemma \ref{l:cycl}.

(ii) Let $A=\prod_{i=1}^{n}\Phi_i(u^2)$. Observe that $\Phi_i(u^2)$ is equal to $\Phi_{2i}(u)$ if $i$ is even and to $\Phi_i(u)\Phi_{2i}(u)$ if $i$ is odd.
Fix $r\in\pi(S)$, set $j=e(r,u^2)$ and choose the largest $s$ such that $jr^s\leqslant n$. Then $(A)_r=(u^{2jr^s}-1)_r$.

Let $r$ be odd. Then $\gexp_r(S)=(u^{jr^s}-1)_r$ or $\gexp_r(S)=(u^{jr^s}+1)_r$ depending on whether $r$ divides $u^j-1$ or $u^j+1$. In any case $(A)_r=\gexp_r(S)$.
If $r=2$, then $\gexp_2(S)=(u^{2^s}-1)_2=(A)_2/2$ if $n\neq 2^s$ and $\gexp_2(S)=(u^{n}-1)_2/2=(A)_2/4$ if $n=2^s$.

The bounds on $\gexp(S)$ follows from Lemma \ref{l:cycl} and the fact that $\gexp_v(S)>2n-1$.

(iii) Let $n$ be even. 
If $r$ is odd and $r$ does not divide $u^n-1$, then $\gexp_r(S)=\gexp_r(O_{2n+1}(u))$. If $r$ is odd and divides  $u^n-1$, then $(u^n-1)_r=(u^{n/2}-1)_r$ or $(u^n-1)_r=(u^{n/2}+1)_r$,
and again $\gexp_r(S)=\gexp_r(O_{2n+1}(u))$. Furthermore, $S$ has elements of order $u^i-1$ for all $i\leqslant n-1$, and so
$\gexp_2(S)=\gexp_2(O_{2n+1}(u))$ for odd $u$.

Let $n$ be odd. If $r$ is coprime to $u^n+1$, then $\gexp_r(S)=\gexp_r(O_{2n-1}(q))$.
Suppose that $r$ is odd and divides $u^n+1$. Then $e(r,u)$ is even and $e(r,n)/2$ divides $n$. If $2n\neq e(r,n)r^s$, then  $\gexp_r(S)=\gexp_r(O_{2n-1}(q))$.
If $2n=e(r,n)r^s$, then $\gexp_r(S)=(q^n+1)_r=\Phi_{2n}(u)_r\gexp_r(O_{2n-1}(u))$. In any case $\gexp_r(S)=(q^n+1)_r=\Phi_{2n}(u)_r\gexp_r(O_{2n-1}(u))$.
Let $u$ be odd. There are elements of $S$ of order $u^i-1$ for every $i\leqslant n-2$ and if $n-1=2^s$, then there is an element of order $u^{n-1}-1$. This yields
$\gexp_2(S)=(\prod_{i=1}^{n-1}\Phi_i(u))_2=
(\prod_{i=1}^{n-1}\Phi_i(u^2))_2/2$.

(iv) Let $n$ be even. 
Reasoning as in (iii), we see that $\gexp_r(S)=\gexp_r(O_{2n-1}(u))$ for odd $r$. If $u$ is odd, then $\gexp_2(S)=(u^{2^s}-1)_2$, where $2^s\leqslant n-2$, and so $\gexp_2(S)=\gexp_2(O_{2n-1}(u))$.

Let $n$ be odd. Reasoning as in (iii), we calculate that $\gexp_r(S)=\Phi_{n}(u)_r\gexp_r(O_{2n-1}(u))$ for all odd $r$. Suppose that $n$ is odd.
If $n-1\neq 2^s$, then $\gexp_2(S)=\gexp_2(O_{2n-1}(u))$.  If $n-1=2^s$, then $S$ has an element of order $u^{n-1}-1$, and so $\gexp_2(S)=2\gexp_2(O_{2n-1}(u))$.

\end{proof}

\begin{lemma}\label{l:exp_e} Let $u$ be a power of a prime $v$.
\begin{enumerate}
\item If $S=E_8(u)$, then $\gexp_v(S)\geqslant 31$ and $$\gexp_{v'}(S)=\frac{(u^{20}+u^{10}+1)(u^{12}+u^6+1)(u^{12}+1)(u^6+1)(u^{20}-1)(u^{14}-1)}{(u^4-1)(5,u^2+1)(3,u^2-1)}.$$ In particular,
$\gexp(S)>2u^{80}$

\item If $S=E_7(u)$, then $\gexp_v(S)\geqslant 19$ and $$\gexp_{v'}(S))=\frac{(u^{12}+u^6+1)(u^{14}-1)(u^{10}-1)(u^{12}-1)(u^4+1)}{(u^2-1)^2(6,u^2-1)}.$$ In particular, $\gexp(S)>3u^{48}$.

\item If $S=E_6^\tau(u)$, then $\gexp_v(S)\geqslant 13$ and $$\gexp_{v'}(S)=\frac{(u^6+\tau u^3+1)(u^5-\tau)(u^{12}-1)}{(u-\tau)(6,u-\tau)}.$$ In particular, $\gexp(S)> u^{22}$.

\item $\gexp({}^2F_4(u))=16(u^6+1)(u^3+1)(u-1)/3$.
\end{enumerate}
\end{lemma}

\begin{proof}
The bounds on $\gexp_v(S)$ follow from \cite[Proposition 0.5]{95Tes}. The formulas for $\gexp_{v'}(S)$ follow from the description of maximal tori of $S$ in \cite{91DerFak, 84Der}.
We prove the lemma for $S=E_7(u)$, the other cases being similar.

Let $S=E_7(u)$ and fix $r\in\pi(S)$, $r\neq v$. Set $j=e(u,r)$ if $r$ is odd and $j=1$ if $r=2$. Then $$j\in I=\{1,2,3,4,5,6,7,8,9,10,12,14,18\}.$$
Define $$A=\prod_{i\in I}\Phi_i(u)=\frac{(u^{12}+u^6+1)(u^{14}-1)(u^{10}-1)(u^{12}-1)(u^4+1)}{(u^2-1)^2}.$$ By Lemma \ref{l:r-part_c}, the number $\Phi_i(u)$ is divisible by $r$ if and only if $i=jr^s$ for $s\geqslant 0$ and, furthermore, $(\Phi_{jr^s}(u))_r=r$ for $s>0$ and odd  $r$. Hence $(A)_r=(\Phi_j(u))_r$ if $r>7$ or $j>2$, $(A)_r=r(\Phi_j(u))_r$ if $r=5,7$ and $j=1,2$, $(A)_r=r^2(\Phi_j(u))_r$ if $r=3$, and $(A)_2=(u^8-1)_2$.

On the other hand, it follows from  \cite{91DerFak} that $\gexp_r(S)=(u^j-1)_r=(\Phi_j(u))_r$ if $r>7$ or $j>2$
If $j=1, 2$ and $r>2$, then $\gexp_r(u)=(u^r\mp 1)_r=r(\Phi_j(u))_r$. Finally, if $r=2$, then $\gexp_r(S)=(u^4-1)_2$. Thus $A=(2,u-1)(3,u^2-1)\gexp_{v'}(S)$.
\end{proof}

\begin{lemma}\label{l:spec_g2}
Let $u$ be a power of a prime $v$. Then $\omega(G_2(u))$ consists of the divisors of the numbers $u^2\pm u+1$,
$u^2-1$, and $v(u\pm 1)$ together with the divisors of
\begin{enumerate}
 \item $8$, $12$ if $v=2$;
 \item $v^2$ if $v=3,5$.
\end{enumerate}
\end{lemma}

\begin{proof}
 See \cite{84Der}, and also \cite[Lemma 1.4]{13VasSt.t}.
\end{proof}

\begin{lemma}\label{l:spec_f4}
Let $u$ be a power of a prime $v$. Then $\omega(F_4(u))$ consists of the divisors of the numbers $u^4-u^2+1$, $u^4+1$, $(u^2\pm u+1)(u^2-1)$, $(u^4-1)/(2,u-1)$,
$v(u^3\pm 1)$, $v(u^2+1)(u\pm 1)$, and $v(u^2-1)$ together with the divisors of
\begin{enumerate}
 \item $4(u^2\pm1)$, $4(u^2\pm u+1)$, $8(u\pm 1)$, and $16$ if $v=2$;
 \item $9(u^2\pm1)$, $27$ if $v=3$;
 \item $25(u\pm 1)$ if $v=5$;
 \item $49\cdot 2$ if  $v=7$;
 \item $121$ if  $v=11$.
\end{enumerate}

\begin{proof}
See \cite{84Der}, and also \cite[Theorem 3.1]{16GrZv}.
\end{proof}

\end{lemma}

\begin{lemma}\label{l:spec_3d4}
Let $u$ be a power of a prime $v$. Then $\omega({}^3D_4(u))$ consists of the divisors of the numbers $u^4-u^2+1$, $(u^2\pm u+1)(u^2-1)$, and
$v(u^3\pm 1)$ together with the divisors of
\begin{enumerate}
 \item $4(u^2\pm u+1)$, $8$ if $v=2$;
 \item $v^2$ if $v=3,5$.
\end{enumerate}
\end{lemma}

\begin{proof}
 See \cite{87DerMich}, and also \cite[Theorem 3.2]{16GrZv}.
\end{proof}

\begin{lemma}\label{l:small_prev}
Let $S$ be one of the groups $S_{2n}(u)$, $O_{2n+1}(u)$, where $n=3$ or $4$, $O_{2n}^\tau(u)$, where
$n=4$ or $5$, or $L_n^\tau(u)$, where $3\leqslant n\leqslant 8$, and let $r\in\pi(S)$. Suppose that $r$ is odd, $r$ does not divide $u$ and
$t(r,S)\geqslant 3$. Then the following holds.
\begin{enumerate}
 \item If $r$ divides $u-1$ or $u+1$, then either $L=L_3^\tau(u)$ and $r\in R_2(\tau u)$, or $L=L_n^\tau(u)$ and $(u-\tau)_r=(n)_r$.
 \item If $r>5$, then a Sylow $r$-subgroup of $S$ is a product of at most two isomorphic cyclic groups unless $L=L_7^\tau(u)$ and $(u-\tau)_7=7$.
\end{enumerate}

\end{lemma}

\begin{proof}
 (i) Suppose that $S=L_n^\tau(u)$, where $u$ is a power of a prime $v$,
 and $r$ divides $u-1$ or $u+1$.  It is clear that $r$ is adjacent in $GK(S)$ to every prime divisor of $u^k-\tau^k$, where $k\leqslant n-2$.
 If $r$ divides $u+\tau$, then $r$ is adjacent to the prime divisors of $u^k-\tau^k$, where $k$ is the even number in the set $\{n,n-1\}$,
 and if $n\neq 3$, then $r$ is adjacent to $v$. Finally, if $r$ divides $u-\tau$, then $r$ is adjacent $v$ and, furthermore, $r$ is adjacent to either $r_n(\tau u)$
  or $r_{n-1}(\tau $ unless $(n)_r=(u-\tau)_r$ by \cite[Propositions 4.2 and 4.3]{05VasVd.t}  If $S$ is a symplectic or orthogonal, the assertion easily follows from the results of \cite{05VasVd.t}.

 (ii)  Let $S=L_n^\tau(u)$ and let $j=e(r,\tau u)$. If $j=1$, then by (i), it follows that $(u-\tau)_r=(n)_r$. Since $n\leqslant 8$ and $r>5$, we have  $r=7$. If $j=2$, then $n=3$ and Sylow
$r$-subgroups of $S$ are cyclic.  If $j\geqslant 3$, then the order of Sylow $r$-subgroups of $S$ is equal to $(u^j-\tau^j)_r\dots (u^{jk}-\tau^{jk})_r$, where $jk$ is the largest multiple of $j$ not
exceeding $n$. It is clear that $k\leqslant 2$, and in particular $(u^{jk}-\tau^{jk})_r=(u^j-\tau^j)_r$. It remains to note that $S$ contains the direct product of $k$ cyclic subgroups of order
$(u^j-\tau^j)_r$. The case where $S$ is symplectic or orthogonal is similar.
\end{proof}

\begin{lemma}\label{l:maut}
Let $S$ be one of the groups $^2B_2(u)$, $^2G_2(u)$, $G_2(u)$, $F_4(u)$, or $E^\tau_6(u)$ and let  $l$ be equal to $1$, $1$, $2$, $4$, or $6$ respectively. Suppose that $a\in\omega(\Aut S)$, $a$ is odd and coprime to $v$.
Then either $a\in\omega(\Inndiag S)$, or $a<6u^{l/3}$, and in any case $a<2u^l$.
\end{lemma}

\begin{proof}
Let  $a\in \omega(\Inndiag S)$. If $S\neq E^\tau_6(u)$, then $a<2u^l$ by \cite[Lemma 1.3]{09VasGrMaz1.t}. If $S=E_6^\tau(u)$, then
$a$ is an element of some maximal tori of $\Inndiag S$. The maximal tori of $\Inndiag S$ are isomorphic to the maximal tori of the universal group corresponding to $S$,
and the structure of the latter tori is found in \cite{91DerFak}. By examining this structure, we derive the desired inequality $a<2u^6$.
% Finally, if  $S=E_6^\tau(u)$ и $(a,v)=v$, то $|a|\leqslant (3,u-\tau)m$, где $m$~---наибольший порядок элемента, кратный $v$, в группе $S$. Из \cite[теорема 1]{13But.t} следует, что $(3,u-\tau)m\leqslant v(u^6-1)/(u-1)<2u^6$.

If $a\not\in\omega(\Inndiag S)$, then there is a field automorphism $\alpha$ of $S$ of odd order $k$ such that $a\in\omega(\alpha\Inndiag S)$.
By \cite[Theorem 2]{17Gr.t}, we have $\omega(\alpha\Inndiag S)=k\cdot \omega(S_0)$, where $S_0$ is a group of the same type as $S$ but over the field of order $u^{1/k}$, and so $a\leqslant k\cdot 2u^{l/k}$. The function
$u_0^x/x$, with $u_0\geqslant 2$, increases with respect to $x$ for $x\geqslant 3$, and hence $2ku^{l/k}\leqslant 6u^{l/3}< 2u^l$.
\end{proof}

\section{Restrictions on the group $G$}

We begin with a result covering a broad class of simple classical groups, not only $S_6(q)$, $O_7(q)$, and $O_8^+(q)$.

% Соглашения: $q\equiv \varepsilon \pmod 4$, $q>5$

%Факты: $R_6(\varepsilon q)\cap \pi(K)=R_6(\varepsilon q)\cap \pi(\overline G/S)=\varnothing$

\begin{lemma} \label{l:a} Let $q$ be odd and let $L$ be one of the simple groups $L_n^\tau(q)$, where $n\geqslant 5$, $S_{2n}(q)$, where $n\geqslant 3$, $O_{2n+1}(q)$, where $n\geqslant 3$, or $O_{2n}^\tau(q)$, where $n\geqslant 4$.
Suppose that $G$ is a finite group such that $\omega(G)=\omega(L)$ and $K$ is the solvable radical of $G$. Then the socle $S$ of $\overline G=G/K$
is a nonabelian simple group and either $K$ is nilpotent, or one of the following holds:
\begin{enumerate}
 \item $L=O_{2n}^\tau(q)$, where $n$ is odd, $q\equiv \tau\pmod 8$, $\pi(K)\cap R_2(\tau q)\neq\varnothing$ and $R_n(\tau q)\cap \pi(S)\subseteq \pi(K)$;
 \item $L=L_n^\tau(q)$, where $1<(n)_2<(q-\tau)_2$, $\pi(K)\cap R_2(\tau q)\neq \varnothing$ and $R_{n-1}(\tau q))\cap\pi(S)\subseteq \pi(K)$;
 \item $L=L_n^\tau(q)$, where either $(n)_2>(q-\tau)_2$ or $(n)_2=(q-\tau)_2=2$, $\pi(K)\cap \pi((q-\tau)/(n,q-\tau))\neq \varnothing$ and $R_{n}(\tau q)\cap\pi(S)\subseteq \pi(K)$.
\end{enumerate}
\end{lemma}

\begin{proof}
Since $t(L)\geqslant 3$ and $t(2,L)\geqslant 2$, it follows that $S$ is a nonabelian simple group by Lemma  \ref{l:str}.
We show that either $G$, $K$ and $S$ satisfy the following condition: for every $r\in\pi(K)$, there is $a\in\omega(S)$ such that $\pi(a)\cap\pi(K)=\varnothing$
and $ar\not\in\omega(G)$, in which case $K$ is nilpotent by Lemma \ref{l:nilp}, or one of assertions (i)--(iii) holds.
If $r=2$, we can take $a$ to be any prime in $\pi(G)$ that is not adjacent to 2 in $GK(G)$. We can assume, therefore, that $r$ is odd.

Let $L$ be one of the groups $S_{2n}(q)$, $O_{2n+1}(q)$, or $O_{2n+2}^+(q)$, where $n\geqslant 3$ is odd.
Choose $\varepsilon\in\{+1,-1\}$ so that $q\equiv \varepsilon \pmod 4$. Then every prime in $R_{2n}(\varepsilon q)$
is not adjacent to 2 in $GK(G)$, and hence $R_{2n}(\varepsilon q)\cap (\pi(K)\cup\pi(\overline G/S))=\varnothing$. Furthermore, if
$m\in \omega(G)$ is a multiple of $r_{2n}(\varepsilon q)$, then $m$ divides $(q^n+\varepsilon)/2$. Thus if
 $r$ is coprime to $(q^n+\varepsilon)/2$, we can take $a=r_{2n}(\varepsilon q)$.
Suppose that $r$ divides $(q^n+\varepsilon)/2$. Then $r$ is coprime to both $(q^n-\varepsilon)/2$ and $q^{n-1}+1$.
Since $\{r_{2n}(\varepsilon q), r_n(\varepsilon q), r_{2n-2}(q)\}$ and $\{r_{2n}(\varepsilon q), r_n(\varepsilon q), p\}$ are cocliques in $GK(G)$, it follows from Lemma \ref{l:str} that
at least one of the sets $R_n(\varepsilon q)$ and $R_{2n-2}(q)\cup\{p\}$ is disjoint from both $\pi(\overline G/S)$ and $\pi(K)$.
Thus we can take $a$ to be $r_n(\varepsilon q)$ or $pr_{2n-2}(q)$ respectively.

Let $L$ be one of the groups $S_{2n}(q)$, $O_{2n+1}(q)$, or $O_{2n}^-(q)$, where $n\geqslant 4$ is even. If $r$ is coprime to $(q^n+1)/2$, then take $a=r_{2n}(q)$.
Suppose that $r$ divides $(q^n+1)/2$. Then it is coprime to $q^{n-1}-1$ and $q^{n-1}+1$, and so $rr_{n-1}(q), rr_{2n-2}(q)\not\in\omega(G)$.
Since $\{r_{2n}(q), r_{2n-2}(q), r_{n-1}(q)\}$ is a coclique in $GK(G)$, at least one of the sets $R_{2n-2}(q)$ and $R_{n-1}(q)$ is disjoint from $\pi(\overline G/S)$ and $\pi(K)$.
Hence we can take $a=r_{2n-2}(q)$ or $a=r_{n-1}(q)$.

Let $L=O_{2n}^\tau(u)$, where $n\geqslant 5$ is odd. The sets $\{r_{n}(\tau q), r_{2n-2}(q), r_{n-2}(-\tau q)\}$
and $\{r_{n}(\tau q), r_{2n-2}(q), p\}$ are cocliques in $GK(G)$, so at most one of the sets $R_n(\tau q)$, $R_{2n-2}(q)$ and $R_{n-2}(-\tau q)\cup\{p\}$ is not disjoint from $\pi(K)\cup\pi(\overline G/S)$.
Also, every number in $\omega(G)$ that is a multiple of $r_n(\tau q)$, $r_{2n-2}(q)$, or $pr_{n-2}(-\tau q)$, has to divide $m_1=(q^n-\tau)/(4,q-\tau)$,  $m_2=(q^{n-1}+1)(q+\tau)/(4,q-\tau)$, or
$m_3=2p(q^{n-2}+\tau)/(4,q-\tau)$ respectively. Note that $(m_1,m_2)=(m_1, m_3)=1$, while $(m_2, m_3)$ divides $q+\tau$.

Suppose that $q\not\equiv \tau\pmod 8$. Then numbers in $R_n(\tau q)$ are not adjacent to 2 in $GK(G)$.
If $r$ is coprime to $m_1$, then take $a=r_n(\tau q)$. If $r$ divides $m_1$, then take $a=r_{2n-2}(q)$ or $a=pr_{n-2}(-\tau q)$, depending on whether
$R_{2n-2}(q)$ or $R_{n-2}(-\tau q)\cup\{p\}$ is disjoint from $\pi(K)\cup\pi(\overline G/S)$. Let $q\equiv \tau\pmod 8$. In this case, the primes not adjacent to 2 in $GK(G)$ are exactly elements of $R_{2n-2}(q)$.
If $r$ is coprime to $m_2$, then take $a=r_{2n-2}(q)$. Assume that $r$ divides $m_2$. If there is $s\in (\pi(S)\cap R_n(\tau q))\setminus \pi(K)$, then take $a=s$. If $\pi(S)\cap R_n(\tau q)\subseteq \pi(K)$
and $r$ is coprime to $m_3$, then $a=pr_{2n-2}(q)$. Thus we are left with the situation where $q\equiv \tau\pmod 8$,  $\pi(S)\cap R_n(\tau q)\subseteq \pi(K)$ and $r$ divides $(m_2,m_3)=q+\tau$, as
required in (i).

Let $L=L_n^\tau(q)$, where $n\geqslant 5$. The sets $\{r_n(\tau q), r_{n-1}(\tau q), r_{n-2}(\tau q)\}$ and $\{r_n(\tau q), r_{n-1}(\tau q), p\}$ are cocliques in $GK(G)$, so
at most one of the sets $R_n(\tau q)$, $R_{n-1}(\tau q)$ and $R_{n-2}(\tau q)\cup\{p\}$ can be not disjoint from $\pi(\overline G/S)\cup\pi(K)$. Every number in $\omega(G)$ that is a multiple of $r_n(\tau q)$, $r_{n-1}(\tau q)$, or   $pr_{n-2}(\tau q)$ divides respectively $m_1=(q^n-\tau^n)/(q-\tau)(n,q-\tau)$,  $m_2=(q^{n-1}-\tau^{n-1})/(n,q-\tau)$, or $m_3=p(q^{n-2}-\tau^{n-2})/(n,q-\tau)$. Observe that $(m_1,m_2)=1$ and $(m_2, m_3)=(q-\tau)/(n,q-\tau)$.

Suppose that $n$ is odd. Then the numbers not adjacent to  2 in $GK(G)$ are precisely elements of $R_n(\tau q)$. Also $(m_1,m_3)=1$ in this case.
If $r$ is coprime to $m_1$, then take $a=r_n(\tau q)$. If $r$ divides $m_1$, then take $a=r_{n-1}(\tau q)$ or $a=pr_{n-2}(q)$, according as $R_{n-1}(\tau q)$ or $R_{n-2}(\tau q)\cup\{p\}$ is disjoint from
$\pi(K)\cup\pi(\overline G/S)$.

Let $n$ be even. If $(n)_2=(q-\tau)_2>2$, then elements of $R_n(\tau q)\cup R_{n-1}(\tau q)$ are not adjacent to $2$ in $GK(G)$, so we can take $a=r_n(\tau q)$ if $r$ is coprime to $m_1$, and $a=r_{n-2}(\tau q)$ otherwise.
Let $(n)_2<(q-\tau)_2$. If $r$ is coprime to $m_1$, then $a=r_n(\tau q)$. If $r$ divides  $m_1$ and there is $s\in (\pi(S)\cap R_{n-1}(\tau q))\setminus \pi(K)$, then $a=s$. If $r$ divides $m_1$ but not $m_3$ and
$\pi(S)\cap R_{n-1}(\tau q)\subseteq \pi(K)$, then $a=pr_{n-2}(\tau q)$. It remains to note in this case that $(m_1,m_3)=(q+\tau)/2$. Let $(n)_2>(q-\tau)_2$ or $(n)_2=(q-\tau)_2=2$. If $r$ is coprime
to $m_2$, then $a=r_{n-2}(\tau q)$.
If $r$ divides $m_2$ and there is $s\in (\pi(S)\cap R_{n}(\tau q))\setminus \pi(K)$, then $a=s$. Finally, if $r$ divides $m_2$ but not $m_3$, then $a=pr_{n-2}(\tau q)$. Thus (ii) or (iii) holds, and the proof is complete.

\end{proof}

Now we begin work toward a proof of Theorem \ref{t:2}. In the rest of the paper, $L$ is one of the groups $S_6(q)$, $O_7(q)$, or $O_8^+(q)$, where $q>3$ is a power of an odd prime $p$, and $G$ is a
finite group isospectral to   $L$. By Lemmas \ref{l:str} and \ref{l:a}, we have that the structure of $G$ is as follows: $$S\leqslant \overline G= G/K\leqslant \Aut S,$$ where $K$ is a normal
nilpotent subgroup of $G$ and $S$ is a nonabelian simple group.
To prove Theorem \ref{t:2}, we assume that $S$ is a group of Lie type in characteristic $v\neq p$ over a field of order $u$.

Fix $\varepsilon\in\{+1,-1\}$ such that $q\equiv \varepsilon \pmod 4$. Then the numbers not adjacent to 2 in $GK(G)$ are precisely elements of $R_6(\varepsilon q)$. By Lemma \ref{l:str}, the set
$R_6(\varepsilon q)$
is disjoint from $\pi(K)\cup\pi(\overline G/S)$, in particular, $k_6(\varepsilon q)\in \omega(S)$.

\begin{lemma}\label{l:rest_k} Suppose that $r\in\pi(K)$ and $r\neq v$.
\begin{enumerate}
 \item If $S$ contains a Frobenius group with kernel a $v$-group and cyclic complement of order $a$, then $ra\in\omega(G)$.
 \item If $S\neq L_2(v)$, then $rv\in \pi(G)$.
 \item If $S\neq L_2(v)$ and $S$ is a classical group, then $r\in R_1(q)\cup R_2(q)$.
 \item If $S\neq L_2(v)$, $s\in R_3(q)\cup R_6(q)$, $s\neq r$ and $s$ divides the order of a proper parabolic subgroup of $S$, then $rs\in \omega(G)$.
\end{enumerate}
\end{lemma}

\begin{proof}
Let $\widetilde G=G/(O_{r'}(K)\Phi(O_r(K))$ and $\widetilde K=K/(O_{r'}(K)\Phi(O_r(K))$. The group $S$ acts on $\widetilde K$ and if this action is not faithful, then $r\cdot\omega_{r'}(S)\subseteq \omega(G)$.
Thus (i) follows from Lemma \ref{l:frob}.

(ii) If $v=2$, then $2r\in\omega(G)$ by Lemma \ref{l:str}, so we can assume that $v$ is odd. Since $S\neq L_2(v)$, there is a noncyclic abelian $v$-subgroup in $S$, and hence $vr\in\omega(G)$ by Lemma \ref{l:nonc}.

(iii) Suppose that $r\not\in R_1(q)\cup R_2(q)$. Then $r\in R_3(\varepsilon q)\cup R_4(q)\cup\{p\}$, and so there is $s\in  R_3(\varepsilon q)\cup R_4(q)\cup\{p\}$ such that
$rs\not\in\omega(G)$. Observe that $s\not\in\pi(K)$, since $\{r,s,r_6(\varepsilon q)\}$ is a coclique in $GK(G)$. Also, $s\neq v$ and $r_6(\varepsilon q)\neq v$ by (ii).
By hypothesis, $S$ is a classical group, therefore, at least one of the numbers $r_6(\varepsilon q)$ and $s$, say $t$, divides the order of the Levi factor of a proper parabolic subgroup $P$ of $S$
(cf. \cite[Lemma 3.8]{15Vas}).
We derive a contradiction by proving that $rt\in\omega(G)$. We can assume that $S$ acts on $\widetilde K$ faithfully.

Let $U$ be the unipotent radical of $P$ and let $g\in P$ be an element of order $t$. By \cite[13.2]{83GorLy}, it follows that $g$ does not centralize $U$.
Applying Lemma \ref{l:hh} to $U\rtimes\langle g\rangle$, we see that either $tr\in\omega(G)$, as desired, or $v=2$, $t$ is a Fermat prime, $U$ is nonabelian
and $C_U(g)\neq 1$ (in particular, $2t\in\omega(S)$). Since both $r_3(q)$ and $r_6(q)$ are congruent to 1 modulo 3, these numbers cannot be Fermat primes, and hence
$t\in R_4(q)\cup\{p\}$ and $r\in R_3(\varepsilon q)$.

If $S\neq U_n(u)$, then the conditions $v=2$ and $2t\in\omega(S)$ imply that $t$ divides the order of some maximal parabolic subgroup with abelian unipotent radical (namely the order of $P_1$ in notation of \cite{90KlLie}).
Suppose that $S=U_n(u)$ and $k=e(t,u)$. Then $k\leqslant n-2$ and $k$ divides $t-1$, and in particular $k=1$ or $k$ is even. If $n\geqslant 4$, then $S$ contains a Frobenius group with complement of order $t$ and
kernel a $v$-group, and so $rt\in\omega(G)$ by Lemma \ref{l:frob}. Let $n=3$. Note that $t$ divides $u+1$ because $2t\in\omega(S)$.  Since $r\in R_3(\varepsilon q)$, we can take any element of $R_4(q)\cup\{p\}$ as $t$, and thus we can assume that every element of $R_4(q)\cup\{p\}$
is a Fermat prime dividing $u+1$. However, $2^b+1$ divides $2^a+1$ if and only if $a/b$ is an odd integer, so $u+1$ cannot be divisible by two different Fermat primes, and we have a contradiction.

(iv) If $s=v$, then $rv\in\omega(G)$ by (ii). If $s\neq v$, then we argue as in (iii) keeping in mind that $s$ is not a Fermat prime.
\end{proof}

\begin{lemma}\label{l:q5}
If $q=5$, then $S$ is not a group of Lie type in characteristic $v\neq 5$.
\end{lemma}

\begin{proof}
Suppose that this is false. Since $k_6(\varepsilon q)=7$, it follows that $7\in\pi(S)\subseteq \{2,3,5,7,13,31\}$. Furthermore, $\{7,13,31\}$ is a coclique in $GK(G)$,
so by Lemma \ref{l:str}, at least one of the numbers $13$ and $31$ lies in $\pi(S)$. Also, it is clear that $\omega(S)\subseteq\omega(G)$.
Using \cite[Table  1]{09Zav}, we check that the only groups of Lie in characteristic $v\neq 5$ satisfying these conditions are
$L_2(13)$, $G_2(3)$, $G_2(4)$, and $^2B_2(8)$. In particular, $31\not\in\omega(\Aut S)$, and hence $31\in\pi(K)$. In $L_2(u)$, $G_2(u)$, or $^2B_2(u)$, there is
a Frobenius subgroup with kernel a $v$-group and cyclic complement of order $(u-1)/(2,u-1)$, $u^2-1$, or $u-1$ respectively (this is a Borel subgroup in $L_2(u)$ and $^2B_2(u)$, and for
$G_2(u)$,  see \cite[Lemma 1.4]{04CaoChenGr.t}). Applying Lemma \ref{l:rest_k}(i), we conclude that $G$ has an element of order $31\cdot 6$, $31\cdot 8$, $31\cdot 15$, or $31\cdot 7$ respectively.
This contradicts the fact that the only multiple of 31 in $\mu(L)$ is $31\cdot 2$.
\end{proof}

\begin{lemma}\label{l:k3v} Let $q>5$.
\begin{enumerate}
\item If $k_4(q)=v^l$, then $l=2$ and $S=L_2(u), L_3^\pm(u), G_2(u),$ or ${}^3D_4(u)$. %$u=v^t$, $t$ нечетно
%\item Если $(v^l-1)_{2'}\in\omega(S)$, то $k_3(q)\neq v^l$ и $k_6(q)\neq v^l$.
\item $k_3(\varepsilon q)\neq v$ and $v\not\in R_6(\varepsilon q)$.
\end{enumerate}
\end{lemma}

\begin{proof} Observe that $v>3$ if $v\in R_4(q)$ or $v\in R_3(q)\cup R_6(q)$, and so $S$
is not a Ree or Suzuki group in this case. In particular, $S$ contains elements of orders $(u\pm 1)/2$, and hence elements of orders $(v\pm1)/2$.

(i) Suppose that $q^2+1=2v^l$. By Lemma \ref{l:nl4}, it follows that either $q=239$ and $v=13$, or $l\leqslant 2$.
Assume that $q=239$ and $v=13$. Since $k_6(\varepsilon q)=k_3(239)=19\cdot 3019\in\omega(S)$ and $e(3019,13)=503$, there is am element of order $61\in R_3(13)$ in $S$.
On the other hand,  $e(61,239)=15$, and so $61\not\in\omega(G)$.

Thus $l=1$ or $l=2$. Using the equality $(v^l+1)/2=(q^2+3)/4$ and the condition $q>5$, it is not hard to check that $(v^l+1)/2$ does not lie in $\omega(G)$.
As we noted, $S$ contains an element of order $(v+1)/2$.
Furthermore, if $S\neq L_2(u)$, $L_3^\tau(u)$, $G_2(u)$, or $^3D_4(u)$, then $S$ contains elements of orders $(u^2+1)/2$ and $(u\pm 1)/2$, and hence elements of order $(v^2+1)/2$.

(ii) Suppose that $k_3(\eta q)=v$ for some $\eta\in\{+,-\}$. Then $v>5$.

If $(3,q-\eta)=3$, then $\omega(G)$ contains $q+2\eta$ because this number is odd and divides $$v-1=k_3(\eta q)-1=(q+2\eta)(q-\eta)/3.$$
Also this number is a multiple of $3$, and so it must divide $(q^3+\eta)/2$ or $(q^4-1)/4$. Since $(q+2\eta,q^3+\eta)=(q+2\eta,9)$ and $(q+2\eta,q^4-1)=(q+2\eta,15)$
it follows that $q+2\eta$ is equal to $9$ or $15$, which yields $q=7$, $11$, $13$, or $17$.
If $q=17$, then $k_3(\eta q)=7\cdot 13$. If $q=13$, then $v=61$ and $(v+1)/2\not\in\omega(G)$.
If $q=11$, then  $v=37$ and $\pi(S)\subseteq \{2,3, 5,7, 11, 19, 37, 61\}$. By \cite[Table 1]{09Zav}, this implies that $S=L_2(37)$,
but then $7\cdot 19=k_6(\varepsilon q)\not\in\omega(S)$. If $q=7$, then  $v=19$ and $\pi(S)\subseteq \{2,3, 5,7,  19, 43\}$,
whence $S=L_2(19)$ or $U_3(19)$. In both cases $43\not\in\pi(S)$ and since $\{5,19,43\}$ is a coclique in $GK(G)$, we have that $S$ contains an element of order $k_4(7)=25$, which is not the case.

If $(3,q-\eta)=1$, there is an element of order $a=(q^2+\eta q+2)/2=(v+1)/2$ in $G$.
It is easy to calculate that $$(q^2+\eta q+2, q+\eta)=(q^2+\eta q+2,q^2+1)=2,$$ $$(q^2+\eta q+2, q-\eta)=(4,q-\eta),$$ and
$$(q^2+\eta q+2, q^2-\eta q+1)=(3\eta q+1,q^2-\eta q+1)=(3\eta q+1,q-4\eta)=(13,q-4\eta).$$
If $a$ divides $(q^3-\eta)/2$ or $(q^4-1)/4$, then $a$ divides $q-\eta$, and hence $a\leqslant 4$. If $a$ divides $(q^3+\eta)/2$, then $a\leqslant 13$.
However $a\geqslant 22$.

Thus $k_3(\varepsilon q)\neq v$ and $k_6(\varepsilon q)\neq v$. If $S\neq L_2(u)$, then $2v\in\omega(S)$, and so $v\not\in R_6(\varepsilon q)$ by Lemma \ref{l:str}. If
$S=L_2(u)$ and $v\in R_6(\varepsilon q)$, then $k_6(\varepsilon q)=v$. This contradiction completes the proof.
\end{proof}

\begin{lemma}\label{l:rest_aut} Let $a\in\omega(\overline G/S)$.
\begin{enumerate}
\item Suppose that $r\in\pi(S)$, $r\not\in\pi(K)\cup \pi(\overline G/S)$ and $rs\not\in\omega(G)$ for all $s\in\pi(a)$. If a Sylow $r$-subgroup of $S$ is a direct product of
$l$ isomorphic cyclic groups, then $a$ divides $\gexp^l_r(G)-1$.
\item If $3$ is coprime to $q+\varepsilon$, then $\pi(a)\subseteq R_1(q)\cup R_2(q)\cup \{p\}$.
If $3$ divides $q+\varepsilon$ and $s\in\pi(a)\setminus R_2(\varepsilon q)$, then either $s=5\in R_4(q)$ and $(a)_5=5$, or $s=7\in R_3(\varepsilon q)$ and $(a)_7=7$.
\item If a Sylow $p$-subgroup of $S$ is a direct product of at most two isomorphic
cyclic groups, then $R_3(\varepsilon q)\cap\pi(\overline G/S)=\varnothing$.
\item Suppose that $p\neq 3$ and for all $r\in R_3(\varepsilon q)\cap \pi(S)$, a Sylow $r$-subgroup of $S$ is a direct product of at most two
isomorphic cyclic groups. Then $(p\cup R_4(q))\cap \pi(\overline G/S)=\varnothing$.
\item If $r\in \pi(\overline G/S)$ and $r\neq 2, v$, then $rv\in\omega(G)$.

\end{enumerate}
\end{lemma}

\begin{proof}
(i) Let $R$ be a Sylow $r$-subgroup of $S$. By hypothesis, the order of $R$ is equal to $(\gexp_r(S))^l=(\gexp_r(G))^l$. By the Frattini argument, $N_{\overline G}(R)$ contains an element of order
$a$,
and this element acts fixed-point-freely on $R$. Thus $a$ divides $|R|-1$.

(ii) Assume that $s\in\pi(a)\setminus R_2(\varepsilon q)$. Denote $k_6(\varepsilon q)$ by $k$ and let $r\in R_6(\varepsilon q)$.
%По лемме \ref{l:str} число $r$ делит порядок группы $S$ и $r$-периоды групп $S$ и $G$ совпадают.
Observe that $r\neq v$ by Lemma \ref{l:k3v}.
Using the fact that $2r\not\in\omega(S)$ and the adjacency criterion in prime graphs of simple groups, we see that Sylow $r$-subgroups of $S$ are cyclic.
By (i), we have $(k)_r\equiv 1\pmod{ (a)_s}$. This congruence holds for every $r\in R_6(\varepsilon q)$, therefore, $(a)_s$ divides $k-1$.

If $3$ divides $q+\varepsilon$, then $k-1=q(q-\varepsilon)$, whence $s\in R_1(\varepsilon q)$ or $s=p$. If $3$ divides $q+\varepsilon$, then $k-1=(q^2-\varepsilon q-2)/3=(q-2\varepsilon)(q+\varepsilon)/3$,
and so $(a)_s$ divides $q-2\varepsilon$. It is clear that $s$ is coprime to both $q-\varepsilon$ and $k$, and hence $(a)_s$ divides $k_4(q)$ or $k_3(\varepsilon q)$.
Since $(q-2\varepsilon,q^2+1)=(q-2\varepsilon, 5)$ and $(q-2\varepsilon, q^2+\varepsilon q+1)=
(7,q^2+\varepsilon q+1)$, the second part of (ii) follows.

(iii) Assume that $s\in R_3(\varepsilon q)\cap\pi(\overline G/S)$. Then $p\in\omega(S)$ and $p\not\in \pi(K)\cup \pi(\overline G/S)$. Applying (i), we have that $s$ divides $\gexp_p(G)-1$ or $\gexp_p^2(G)-1$.
Since the $p$-exponent of $G$ is at most $p^2$, it follows that $s$ divides $p^4-1$. This contradicts our assumption that $s\in R_3(\varepsilon q)$.

(iv) Suppose that $s\in (p\cup R_4(q))\cap\pi(\overline G/S)$. Then $R_3(\varepsilon q)\cap(\pi(K)\cup \pi(\overline G/S))=\varnothing$. Denoting $k_3(\varepsilon q)$ by $k$
and reasoning as in (ii), we see that $s$ divides $k^2-1$.

Assume that $3$ divides $q-\varepsilon$. Then $k-1=(q^2+\varepsilon q-2)/3$ and $k+1=(q^2+\varepsilon q+4)/3$. Both these numbers are coprime to $p$. Also, $s\not\in R_4(q)$ by (ii).
If $3$ does not divide $q-\varepsilon$, then  $k-1=q(q+\varepsilon)$ and $k+1=q^2+\varepsilon q+2$. Both these numbers are coprime to $k_4(q)$, and $s\neq p$ by (ii).

(v) Let $G/S$ contain an element of order $r$. Since $v$ divides the order of the centraliser of any field automorphism of $S$, we can assume that this element is not an image of a field
automorphism.
Then either $r=3$ and $S$ is one of the groups $D_4(q)$, $E_6(q)$, $^2E_6(q)$, or $r$ divides $q-\tau$ and $S=L_n^\tau(u)$, with $n>2$. In the former case $3v\in\omega(S)$ by \cite[Proposition 3.2]{05VasVd.t}.
In the latter case, applying \cite[Proposition 3.1]{05VasVd.t}, we have that $rv\in\omega(S)$ unless $L=L_3^\tau(u)$, $r=3$ and $(u-\tau)_3=3$.
In this situation,  $u$ is not a cube since  $(u-\tau)_3=3$, and hence $\overline G$ includes  $PGL_3^\tau(u)$. But then $v(u-\tau)\in\omega(\overline G)$, as required.
\end{proof}

\section{The case of classical groups}

In this section we show that $S$ is not a classical group. Recall that $S$ is a group over a field of order $u$ and characteristic $v$ and that we chose $\varepsilon\in\{+1,-1\}$ so that
$q\equiv \varepsilon \pmod 4$.
According to Lemma \ref{l:q5}, we can assume that $q\geqslant 7$, and so $k_3(\varepsilon q), k_6 (\varepsilon q)\geqslant 19$.

\begin{lemma}\label{l:l2} $S\neq L_2(u)$.

\end{lemma}

\begin{proof}
Let $S=L_2(u)$. Then $\omega(S)$ consists of all divisors of the numbers $$v, (u-1)/(2,u-1) \text{, and } (u+1)/(2,u-1).$$ By Lemma \ref{l:k3v}, neither
$k_6(\varepsilon q)$, nor $k_3(\varepsilon q)$ equals $v$. Thus $k_6(\varepsilon q)$ divides $(u+\eta)/(2,u-1)$ for some $\eta\in\{+1,-1\}$.
Denote by $F$ a Frobenius subgroup of $S$ with kernel of order $u$ and cyclic complement of order $(u-1)/(2,u-1)$.

Assume that $p\in\pi(\overline G/S)$ or $p\in\pi(K)$. Then $R_3(\varepsilon
q)\cap(\pi(K)\cup \pi(\overline G/S))=\varnothing$, and therefore $k_3(\varepsilon
q)\in\omega(S)$,
which implies that $k_3(\varepsilon q)$ divides $(u-\eta)/(2,u-1)$. If $p=3$,
then 3 divides $u+\eta$ or $u-\eta$, and so $p$ is adjacent to $r_3(q)$ or $r_6(q)$ in $GK(S)$, a contradiction.
Now by Lemma \ref{l:rest_aut}(iv), it follows that $p\not\in\pi(\overline G/S)$.
Hence $p\in\pi(K)$. Applying Lemma \ref{l:rest_k}(i) to $F$,
we have that one of the numbers $pk_3(\varepsilon q)$ and $pk_6(\varepsilon q)$
lies in $\omega(G)$, a contradiction.

Thus $p\not\in\pi(K)\cup\pi(\overline G/S)$, and so $p$ divides $(u-\eta)/(2,u-1)$. Again by Lemma \ref{l:rest_k}(i), we have $R_3(\varepsilon q)\cap \pi(K)\subseteq \{v\}$. Furthermore, Lemma
\ref{l:rest_aut}(iii) implies that $R_3(\varepsilon q)$ is disjoint from $\pi(\overline G/S)$. If $R_3(\varepsilon q)$ contains a prime number $s$ not equal to $v$,
then $s\in \pi(S)$ and $s$ is adjacent to $p$ or $r_6(\varepsilon
q)$ in $GK(S)$, which is a contradiction.

So we can assume that $k_3(\varepsilon q)=v^l$, where $l>1$, and $v\geqslant 7$.
Since $v\not\in\omega(\overline G/S)$, we have  $v\in\pi(K)$. Then
$k_4(q)\in\omega(S)$, and hence $(q^2+1)/2$ divides $(u-\eta)/2$.
So $p(q^2+1)/2$ divides $(u-\eta)/2$. One of the numbers $p(q^2+1)/2$ and
$p(q^2+1)$ lies in $\mu(G)$. The number $(u-\eta)/2$ is even, so
\begin{equation}\label{e:l2}(u-\eta)/2=p(q^2+1).\end{equation}
Then $$(u+\eta)/2=p(q^2+1)+\eta.$$
Recall that $(q^2-\varepsilon q+1)/(3,q+\varepsilon)$ divides $(u+\eta)/2$. If
$3$ divides $q+\varepsilon$, then $p\neq 3$, so $3$ does not divide $u-\eta$ by
\eqref{e:l2}, and therefore $3$ divides $u+\eta$. Hence $q^2-\varepsilon q+1$
divides $(u+\eta)/2$.
But
$$(q^2-\varepsilon q+1,p(q^2+1)+\eta)=(q^2-\varepsilon q+1,
pq+\varepsilon\eta)=(q^2-\varepsilon q+1, q+\frac{\eta q}{p}-\varepsilon)$$ and
$q+q/p+1<q^2-q+1$, a contradiction.
\end{proof}

Our next step is to consider classical groups of not very large dimensions.

\begin{lemma}\label{l:small} $S$ is not one of the following groups:
$S_{2n}(u)$, $O_{2n+1}(u)$, where $n\leqslant 4$, $O_{2n}^\tau(u)$, where
$n\leqslant 5$, $L_n^\tau(u)$, where $3\leqslant n\leqslant 8$.
\end{lemma}

\begin{proof}
Assume the contrary. Since $k_3(\varepsilon q)$ is larger than $7$, Lemma \ref{l:rest_aut}(ii)
implies that $k_3(\varepsilon q)\not\in\omega(\overline G/S)$. So
there exists a number $r'_3(\varepsilon q)\in R_3(\varepsilon q)$ that lies in
$\pi(S)\cup\pi(K)$. Also by Lemma \ref{l:rest_aut}(ii), the numbers $p$ and
$r_4(q)$ cannot both lie in $\pi(\overline G/S)$, and hence there exists a
number $s\in \{p\}\cup R_4(q)$ that  lies in $\pi(S)\cup\pi(K)$.
By Lemma \ref{l:rest_k}(iii), we have  $\{s, r'_3(\varepsilon
q)\}\cap\pi(K)\subseteq \{v\}$, and therefore
$\{s, r'_3(\varepsilon q), r_6(\varepsilon q)\}$ is a coclique in $GK(S)$. In
particular, $t(S)\geqslant 3$, and so $S\neq S_4(u)$.

We claim that $p\neq 3$. Otherwise $\{3, r'_3(\varepsilon q),  r_6(\varepsilon
q)\}$ is a coclique in $GK(S)$ not containing $2$. On the other hand, since $v\neq 3$,
it follow that $3$ divides $u-1$ or $u+1$. By Lemma \ref{l:small_prev}(i), we
have that either $S=L_3^\tau(u)$ and $3$ divides $u+\tau$, or $S=L_3^\tau(u), L_6^\tau(u)$ and $(u-\tau)_3=3$.
If $S=L_6^\tau(u)$ and $(u-\tau)_3=3$, or $S=L_3(\tau u)$ and $(u+\tau)_3>3$,
then $\omega(S)$ contains a number of the form $9k$, where $k>1$, unless $S=L_3(8)$:
this is $v(u^3-\tau)$ in the first case and $u+\tau$ in the second. This leads to a contradiction
since $9\in\mu(G)$ and there are no cocliques of size $3$ consisting of odd primes in $GK(L_3(8))$.
Thus we can assume that $S=L_3^\tau(u)$, where $(u-\tau)_3=3$ or $(u+\tau)_3=3$. In either case,
$u$ is not a cube and $\gexp_3(S)=3$. Since $p\not\in\pi(K)$ by Lemma \ref{l:rest_k}(iii), it follows that $3\in \pi(\overline
G/S)$. Hence $(u-\tau)_3=3$, $\overline G$ includes $\Inndiag S$ and $r_3'(\varepsilon q)\neq v$ by Lemma \ref{l:rest_aut}(v).
It follows that $\overline G$ has an element of order $u^2-1$ and one of the numbers $r'_3(\varepsilon q)$ and $r_6(\varepsilon q)$
divides $u^2-1$. This contradicts the fact that $\{3, r'_3(\varepsilon q),  r_6(\varepsilon
q)\}$ is a coclique in $GK(G)$.

Suppose that the intersection of $\pi(\overline G/S)$ and
$R_4(q)\cup\{p\}$ is not empty. Then $R_3(\varepsilon q)$ is disjoint from
$\pi(K)\cup \pi(\overline G/S)$, and so $\{s, r_3(\varepsilon q),
r_6(\varepsilon q)\}$ is a coclique in $GK(S)$ for every $r_3(\varepsilon q)$.
Using he mentioned above properties of Sylow subgroups of $S$ and applying  Lemma \ref{l:rest_aut}(iv),
we derive a contradiction unless $7\in R_3(\varepsilon q)$ and
$S=L_7^\tau(u)$. In this case, if $t$ lies in the intersection of $\pi(\overline G/S)$ and
$R_4(q)\cup\{p\}$, then $t$ can be the order of a field automorphism only, and
since $7\in\pi(L_7^\tau(u_0))$ for every $u_0$, we obtain that
$7t\in\omega(\overline G)$, a contradiction.

Thus $p\in\pi(S)$, $R_4(q)\subseteq \pi(S)$, and $\{p,r_3'(\varepsilon q),
r_6(\varepsilon q)\}$, $\{r_4(\varepsilon q),r_3'(\varepsilon q),
r_6(\varepsilon q)\}$ are cocliques in $GK(S)$. As we remarked, if $S\neq L_n^\tau(u)$, then
these cocliques do not contain divisors of $u\pm 1$.
If $k_4(q)\neq v^l$, then define $r_4'(q)$ to be any element of
$R_4(q)\setminus\{v\}$. If $k_4(q)=v^l$, then $S=L_3^\tau(u)$ by Lemma
\ref{l:k3v}, and we take $r_4'(q)=v$. In both cases $pr_4'(q)\in \omega(S)$. Let
$a_4$, $a_3$ and $a_6$ be the numbers in $\mu(S)$ divisible by $pr_4'(q)$,
$r_3'(\varepsilon q)$ and $r_6(\varepsilon q)$ respectively.  Observe that $a_6$ is coprime to $v$ and
if $L\neq L_3^\tau(u)$, then $a_4$ is coprime to $v$ too.  It is clear that
$a_4$ divides $p(q^2+1)$, while  $a_3$ and $a_6$ divide $(q^3-\varepsilon)/2$ and
$(q^3+\varepsilon)/2$ respectively. Hence \begin{equation}\label{e:small} \text{
for any distinct }i,j\in\{4,3,6\}\text{, the number }(a_i,a_j)\text{ divides
}2.\end{equation}

We will show that \eqref{e:small} can hold only for finitely many $S$. In these remaining
cases, we will calculate the numbers in $\omega(S)$ that can be divisible by $k_6(\varepsilon q)$
and then for every divisor $h$ of these numbers, solve the equation $k_6(\varepsilon q)=h$ using Lemma \ref{l:k3}. This
will give in turn finitely many possible $q$.

Let $S=S_6(u)$, $O_7(u)$, or $O_8^+(u)$. Then $u>2$, because otherwise $t(S)=2$.
The number $a_6$ is equal to $(u^3+\eta)/(2,u-1)$ for some $\eta\in\{+1,-1\}$.
By \cite[Table 3]{11VasVd.t}, it follows that either $r_3'(\varepsilon q)\in R_3(\eta u)$ and $r'_4(q),p\in R_4(u)$, or
$r'_4(q),p\in R_3(\eta u)$ and  $r_3'(\varepsilon q)\in R_4(u)\cup\{v\}$.
Hence one of the numbers $a_3$ and $a_4$ is equal to $(u^3-\eta)/(2,u-1)$, and
we can take another one to be a multiple of $(u^2+1)(u+1)/(2,u-1)$ or $v(u^2-1)/(2,u-1)$. Applying
\eqref{e:small}, we have that $(u+1)/(2,u-1)$ divides $2$, whence $u=3$. Then
$k_6(\varepsilon q)$ divides $k_3(3)=13$. This is a contradiction since $k_6(\varepsilon
q)\geqslant 19$.

Let $S=S_8(u)$, $O_9(u)$, or $O_8^-(u)$. The maximal under divisibility orders of semisimple elements of
$S$ are equal to $(u^4\pm 1)/(2,u-1)$ and $(u^2\pm u+1)(u^2-1)/(2,u-1)$. If $r_3'(\varepsilon q)=v$, then
we can take $a_3$ to be $v(u^2-1)$. Hence there are two numbers among $a_3$, $a_4$, $a_6$ that are divisible by
$(u^2-1)/(2,u-1)$. Thus $u^2-1$ divides $2(2,u-1)$, which is impossible.

Let $S=O_{10}^\tau(u)$. We claim that for every $i\in\{3,4,6\}$, we can choose $a_i$
so that it is a multiple of $(u-\tau)/(4,u-\tau)$ or $(u+\tau)(2,u-1)$.
If $r_3'(\varepsilon q)=v$, then we take $a_3$ to be a multiple of $v(u^2-1)$.
Now assume that all the numbers $a_3$, $a_4$ and $a_6$ are the orders of semisimple elements, so
each of them has the form $a=[u^{i_1}-\varepsilon_1, \dots, u^{i_s}-\varepsilon_s]/c$, where
$i_1+\dots+i_s=5$, $\varepsilon_1\dots\varepsilon_s=\tau$ and $c$ is as follows: $c=1$ if $s>2$, $c\leqslant (2,u-\tau)$ if $s=2$, and
$c=(4,u-\tau)$ if $s=1$. If at least two of $i_1,\dots,i_s$, say $i_1$ and $i_2$ are equal
to 1, then we can replace $a$ by $[u-1, u-\varepsilon_1\varepsilon_2, u^{i_3}-\varepsilon_3,\dots,u^{i_s}-\varepsilon_s]$.
We can also assume that none of the numbers $u^{i_j}-\varepsilon_j$ are equal to $u^2-1$ or $u^4-1$,
and that these numbers are pairwise distinct. In particular, this yields $s\leqslant 2$. Then
we are left with $(u^5-\tau)/(4,u-\tau)$,  $(u^4+1)(u+\tau)/(4,u-\tau)$ and
$(u^2+1)(u^3+\tau)/(4,u-\tau)$, which satisfy the required property. Thus
there are two numbers among $a_3$, $a_4$ and $a_6$ that are both divisible by $(u-\tau)/(4,u-\tau)$ or $(u+\tau)(2,u-1)$.
Applying \eqref{e:small}, we see that
$u-\tau$ divides 8 or $u+\tau$ divides 4. Hence $u\leqslant 5$,
or $u=7$ and $\tau=-$,  or $u=9$ and $\tau=+$.
Calculating the numbers in $\mu(S)$ that can be divisible by
$k=k_6(\varepsilon q)$ and discarding their prime divisors that are not congruent to 1 modulo
3, we conclude that $k$ is one of the numbers $31$, $61$, $193$, $313$, and $1201$. This
is impossible by Lemma \ref{l:k3}(vi,v).

Let $S=L_n^\tau(u)$, where $3\leqslant n\leqslant 8$. Then every number in $\mu(S)$
except for $(u^n-\tau^n)/(u-\tau)(n,u-\tau)$ and possibly some power of $v$
is a multiple of $(u-\tau)/(n,u-\tau)$. It is clear that $a_4$ and $a_6$ are not powers of $v$,
and we can also choose $a_3$ to be not a power of $v$.
Thus there are two numbers  among $a_3$, $a_4$ and $a_6$ that are multiples of $(u-\tau)/(n,u-\tau)$
and by \eqref{e:small}, it follows that $u-\tau$ divides $2n$. In particular,
$u\leqslant 17$. Let $k=k_6(\varepsilon q)$. Reasoning as above, we have that either $k\leqslant 601$,
or the situation is impossible by Lemma \ref{l:k3}(v). If $k\leqslant 601$, then by Lemma \ref{l:k3}(iv),
there are the following possibilities: $S=U_6(11)$ with $k=19$ and $q=7$; $S=L_7(2),
L_8(2)$ with $k=127$ and $q=19$; $S=L_7(8)$ with $k=19,127$ and $q=7,19$; $S=L_8(9),
U_8(3)$ with $k=547$ and $q=41$. It is easy to check that $\pi(S)\not\subseteq\pi(G)$ in all cases.
\end{proof}

It is clear that $\gexp(S)$ divides $\gexp(G)$, and hence $q$ is
bounded from below in terms of $u$. On the other hand, $k_6(\varepsilon q)$ divides $k_j(u)$ for some $j$, and so $q$ is bounded from above
in terms of $u$. We will show that these two bounds are incompatible for all remaining classical groups $S$.
Recall that  $k_6(\varepsilon q)\geqslant q^2/4+1$
and $k_3(\varepsilon q)>q^2/4$ by Lemma \ref{l:k3},
and $\gexp(G)<q^9$ by Lemma \ref{l:expc3}. Also recall that $F(n)=\sum_{i=1}^{n}\varphi(i)$. The following table
displays the values of this function for small $n$.

$$\begin{array}{|l|c|c|c|c|c|c|c|c|c|c|}
  \hline
  n& 5 & 6& 7 & 8& 9&10&11&12&13&14\\
  \hline
  F(n)& 10&12&18&22&28&32&42&46&58&64\\
  \hline
  n& 15 & 16& 17 & 18& 19&20&21&22&23&24\\
  \hline
  F(n)& 72&80&96&102&120&128&140&150&172&180\\
  \hline
  \end{array}
$$

\begin{lemma}\label{l:lin}
 $S\neq L_n^\tau(u)$, where $n\geqslant 9$.
\end{lemma}

\begin{proof}
Let $S=L_n^\tau(u)$, where $n\geqslant 9$.
By Lemma \ref{l:exp}, we have 
$$\gexp(S)\geqslant \frac{n}{c}\cdot \prod_{i=1}^n\Phi_i(\tau u)>
\frac{n}{c}\cdot u^{3F(n)/4},$$
where $c=r$ if $r\in\pi(u-\tau)$ and $n=r^s$, and $c=1$ otherwise.

The number $k_6(\varepsilon q)$ divides $k_n(\tau u)$ or
$k_{n-1}(\tau u)$, and since both these numbers do not exceed $2u^{n-1}$, we have $q^2/4<2u^{n-1}$.
Hence $$u^{3F(n)/4}<\gexp(S)\leqslant \gexp(G)\leqslant
q^9<8^{9/2}u^{9(n-1)/2}.$$
This yields  $3F(n)/4<27/2+9(n-1)/2$, whence $F(n)<6n+12$.
Applying Lemma \ref{l:sum_varphi}, we see that  $[(n+1)/2]^2<6n+12$, and so
$n\leqslant 24$.
Furthermore, the values of $F(n)$ for $21\leqslant n\leqslant 24$ show that
$n\leqslant 20$.

Now using more precise estimates of $\gexp(S)$ and $k_n(\tau u)$, $k_{n-1}(\tau u)$,
we show that $n\leqslant 14$.

If $n=20$, then $\gexp(S)\geqslant 20u^{3F(n)/4}$ and $q^2/4<2u^{n-2}$, so
$$3F(n)/4<8^{9/2}u^{9(n-2)/2}/20<2^{19/2}u^{9(n-2)/2}.$$ Hence
$3F(n)/4<19/2+9(n-2)/2$, which yields $F(n)<6n+2/3$. This is a contradiction because $F(20)=128$.

If $n=15,16$, then $\gexp(S)\geqslant 8u^{3F(n)/4}$ and $q^2/4<2u^{8}$. It follows that 
$3F(n)/4<(27/2-3)+ 36$, and so $F(n)<62$. But $F(15)=72$, a
contradiction.

Let $n=17,19$. If $n=19$, then $\prod_{i=1}^n\Phi_i(\tau u)>u^{F(n)}/2^6$ since
$\Phi_i(u)\Phi_i(-u)>u^{2\varphi(i)}$ for $i=3$, $5$, $7$, or $9$; $\Phi_i(\tau
u)>u^{\varphi(i)}$ for $i=4$, $8$, or $16$; $\Phi_i(\tau u)>u^{\varphi(i)}/2$ for
$i=11$, $12$, $13$, $15$, or $17$; and
$\Phi_{19}(\tau)(u^2-1)=(u^{19}-\tau)(u+\tau)>u^{20}/2$. A similar inequality
holds if $n=17$. Thus $u^{F(n)}<2^6\cdot 2^{27/2}u^{9(n-1)/2}$, whence
$F(n)<9n/2+15$, which is not true.

Thus $n\leqslant 14$. We claim that
\begin{equation} \label{e:lin}
 \gexp(S)\geqslant 40u^{3(n-1)},
\end{equation}
unless $n=9$ and $u\leqslant 3$. 
Bounding $\prod_{i=1}^n\Phi_i(\tau u)$ as in the case $n=19$, we calculate
that \begin{equation}\label{e:lin1} \gexp(S)>u^{F(n)}/32\text{ if }n\geqslant
11,\end{equation} \begin{equation}\label{e:lin2} \gexp(S)>10u^{F(n)}/4\text{ if
}n=10,\end{equation}
 and \begin{equation} \label{e:lin3} \gexp(S)>3u^{F(n)}/8\text{ if
}n=9.\end{equation} If $n\geqslant 11$, then $F(n)-3(n-1)\geqslant 12$, so
$u^{F(n)-3(n-1)}\geqslant 2^{12}>32\cdot 40$. If $n=10$, then $F(n)-3(n-1)=5$,
and hence $u^{F(n)-3(n-1)}=u^5> 40\cdot 2/5$. If $n=9$, then $F(n)-3(n-1)=4$ and therefore
$u^{F(n)-3(n-1)}=u^4> 40\cdot 8/3$ for $u>3$. Thus \eqref{e:lin} holds.

Suppose that $b=(q^3+\varepsilon)/2\in \omega(\overline G)$. Then $b\leqslant
2u^{n-1}$ by \cite[Table 3]{15GMPS}, and Lemma \ref{l:exp} implies
that
$$\gexp(S)\leqslant \gexp(G)<5b^3\leqslant 40u^{3(n-1)},$$ which contradicts
\eqref{e:lin}. Hence $r_2(\varepsilon q)\in\pi(K)$. By nilpotence of $K$,
we have  $R_3(\varepsilon q)\cap \pi(K)=\varnothing$. Let $r\in
R_3(\varepsilon q)\cap \pi(S)$. If $v\not\in R_2(\varepsilon q)$, then  $r$
cannot divide the order of a proper parabolic subgroup of $S$ by Lemma
\ref{l:rest_k}(iv), and if $v\in R_2(\varepsilon q)$, then $vr(\varepsilon
q)\not\in\omega(S)$. In either case,
$r$ divides $k_n(\tau u)$ or $k_{n-1}(\tau u)$. Since $7\in \pi(S)$ and $7$ does
not divide $k_n(\tau u)k_{n-1}(\tau u)$, we have  $7\not\in
R_3(\varepsilon q)$. Now applying Lemma \ref{l:rest_aut}(ii), we conclude that
$k_3(\varepsilon q)$ divides $k_n(\tau u)$ or $k_{n-1}(\tau u)$.

Let $i$ be the even number in the set $\{n, n-1\}$. Then $\varphi(i)=6$ if
$n=14$, and $\varphi(i)=4$ if $9\leqslant n\leqslant 13$. By the above reasoning,
one of the numbers $k_3(\varepsilon q)$ and $k_6(\varepsilon q)$ divides $k_i(u)$,
and so $q^2/4<k_i(u)<2u^{\varphi(i)}$.
If $n=14$, then $q^2/4<2u^6$ and $F(n)=64$, which yields $$\gexp(S)<q^9\leqslant
2^{27/2}u^{27}< u^{41}<u^{F(n)}/32.$$ If $n=11,12,13$, then $F(n)\geqslant 42$,
and we  have that $$\gexp(S)<q^9<2^{27/2}u^{18}<u^{32}<u^{F(n)}/32.$$ If
$n=10$, then $\gexp(S)<u^{32}=u^{F(n)}$. Finally, if $n=9$ and $u>3$, then
$$\gexp(S)<q^9<2^{27/2}u^{18}<u^{25}<3u^{F(n)}/8.$$ The derived inequalities 
contradict \eqref{e:lin1}--\eqref{e:lin3}.

We are left with the case $S=L_9^\tau(u)$, where $u=2,3$. Let $u=2$. Since
$k_9(2)=73$, $k_9(-2)=19$ and $k_8(2)=17$,  it follows that $k_6(\varepsilon
q)=73$ or $k_6(\varepsilon q)=19$. Then $q=9$ or $q=7$ by Lemma \ref{l:k3}. In
either case, $17\in\pi(S)\setminus\pi(G)$. Let $u=3$. Then
$k_6(\varepsilon q)$ divides $k_9(3)=757$ or $k_9(-3)=19\cdot 37=703$. Lemma
\ref{l:k3} implies that $\tau=-$, $k_6(\varepsilon q)=19$ and $q=7$, and so
$37\in\pi(S)\setminus\pi(G)$. This contradiction completes the proof.
\end{proof}

\begin{lemma}\label{l:orth}
 $S\neq O_{2n}^\tau(u)$, where $n\geqslant 7$ is odd.
\end{lemma}

\begin{proof}
Let $S=O_{2n}^\tau(u)$, where $n\geqslant 7$ is odd. Then $k_6(\varepsilon q)$
divides $k_n(\tau u)$ or $k_{2n-2}(u)$. In either case, $q^2/4\leqslant 2u^{n-1}$.
It follows that $$\gexp(S)\geqslant \frac{2n-1}{2}\cdot \Phi_n(\tau
u)\prod_{i=1}^{n-1}\Phi_i(u^2)>\frac{2n-1}{2}\cdot u^{3(F(n)+F(n-1))/4}.$$

Let $n\geqslant 11$. Then
$$ 10u^{3(F(n)+F(n-1))/4}<\gexp(S)\leqslant \gexp(G)<q^9<8^{9/2}u^{9(n-1)/2},$$
whence
$F(n)+F(n-1)<6n+8$. Since $F(n)+F(n-1)\geqslant n^2/2$, we conclude that
$n\leqslant 13$. Calculating $F(n)+F(n-1)$ for $n=13,11$ yields $n\geqslant 9$.

Let $n=9$. Then $q^2/4\leqslant u^8$. Also $\gexp(S)>17/2\cdot
u^{F(n)+F(n-1)}/4>2u^{50}$ since $\Phi_i(u^2)>u^{2\varphi(i)}$ for
$i=2,4,5,7$, or $8$,
$\Phi_3(u^2)\Phi_6(u^2)>u^{8}$, $\Phi_9(\tau u)>u^6/2$ and $\Phi_1(u^2)>u^2/2$.
Thus $2u^{50}<2^9u^{36}$, which is a contradiction.

Let $n=7$ and $u\geqslant 8$. Then $k_7(\tau u)<u^7/(u-1)\leqslant 8u^6/7$. It
is clear that $k_{12}(u)<u^4$. Thus $q^2/4<8u^6/7$.
Furthermore, $\gexp(S)\geqslant 13u^{30}/4$. It follows that 
$13u^{30}/4<(32/7)^{9/2}u^{27}$, whence $u^3<287$, a contradiction.

It remains to deal with the case when $n=7$ and $u\leqslant 7$. Reasoning as in the proof of 
Lemma \ref{l:small}, we calculate possible values for $k=k_6(\varepsilon q)$ and then apply Lemma \ref{l:k3}
to determine $q$. This shows that either $S=O_{14}^+(2)$, $O_{14}^+(4)$, $k=127$ and $q=19$, or
$S=O_{14}^-(3)$, $k=547$ and $q=41$. Then $\pi(S)\setminus\pi(G)$ contains 17 or
61. This contradiction completes the proof.

\end{proof}

\begin{lemma}\label{l:sym}
 $S\neq S_{2n}(u), O_{2n+1}(u)$, where $n\geqslant 5$, and $S\neq
O^\pm_{2n}(u)$, where $n\geqslant 6$ is even.
\end{lemma}

\begin{proof}
Assume the contrary. For convenience, we rewrite the group $O_{2n}^+(q)$, where $n\geqslant 6$ even,
as $O_{2n+2}(q)$, where $n\geqslant 5$ is odd. 
Then $$\gexp(S)\geqslant\frac{2n-1}{c(2,u-1)}\cdot
\prod_{i=1}^n\Phi_i(u^2)>\frac{2n-1}{4}u^{3F(n)/2},$$
where $c=(2,u-1)$ if $n=2^s$ and $c=1$ otherwise.

Since $k_6(\varepsilon q)$ divides one of the numbers $k_{2n}(u)$, $k_{n}(u)$,
$k_{2n-2}(u)$, and $k_{n-1}(u)$, we have  
$q^2/4\leqslant k_6(\varepsilon q)-1\leqslant u^n$.
Hence $u^{3F(n)/2}<2^9u^{9n/2}$, which yields $F(n)< 3n+6$. The inequality
$F(n)\geqslant [(n+1)/2]^2$ forces $n\leqslant 12$. Furthermore, $F(12)=46$,
and so $n\leqslant 11$.

Let $n=9,10,11$. Then $\gexp(S)>17/2\cdot u^{2F(n)}/8>u^{2F(n)}$, and so
$2F(n)<9+9n/2$, which is not true.

Let $n=8$. Then $\gexp(S)>15/(2,u-1)^2\cdot u^{2F(n)}/2>4u^{44}/(2,u-1)^2$ and
$q^2/4\leqslant u^8/(2,u-1)$,
and hence $4u^{44}/(2,u-1)^2<2^9u^{36}/(2,u-1)^{9/2}$. This yields $u^8<2^7$, which is 
impossible. 

Let $n=7$. Then $\gexp(S)>13/2\cdot u^{2F(n)}/2>2^{3/2}u^{36}$ and
$q^2/4\leqslant 2u^6$. It follows that $2^{3/2}u^{36}<2^{27/2}u^{27}$,
whence $u^9<2^{12}$, and so $u=2$. Then $k_6(\varepsilon q)=127$ or
$k_6(\varepsilon q)=43$. Lemma \ref{l:k3} shows
that $q=19$. But then $17\in\omega(S)\setminus\omega(G)$.

Let $n=6$. Then $\gexp(S)>11/(2,u-1)\cdot u^{2F(n)}/2=11u^{24}/2(2,u-1)$ and
$q^2/4\leqslant 2u^4$. This yields 
$u^6<(2,u-1)2^{29/2}/11$, which forces $u=2,3$. If
$u=2$, then $k_6(\varepsilon q)=31$, which is impossible.
If $u=3$, then $k_6(\varepsilon q)$ is equal to one of the numbers $73$ and
$61$, and so $q=9$ by Lemma \ref{l:k3}. This is a contradiction because $p\neq v$.

Now let $n=5$. Assume that $u\geqslant 17$. Then $\gexp(S)\geqslant 17u^{20}/4$
and $q^2/4\leqslant u^5/(u-1)\leqslant 17u^4/16$.
It follows that $17u^{20}/4< (17/4)^{9/2}u^{18}$, whence $u^2<159$, a
contradiction.
For $u\leqslant 13$, we calculate all possible values of $k_6(\varepsilon q)$ as we did 
previously and then apply Lemma \ref{l:k3} to deduce that $u=8$,
$k_6(\varepsilon q)=331$ and $q=31$. Then $11\in\pi(S)\setminus \pi(G)$, and this contradiction completes
the proof. 
\end{proof}

\section{The case of exceptional groups}

In this section we show that $S$ is not an exceptional group of Lie type. We continue 
to assume that $u$ is a power of a prime $v$ and $q\equiv \varepsilon \pmod 4$.

\begin{lemma}
 $S\neq{}^2B_2(u)$.
\end{lemma}

\begin{proof}
Let $S={}^2B_2(u)$, where $u=2^{2m+1}\geqslant 8$. It is well known that
$\mu(S)=\{4, u-1, u+\sqrt{2u}+1, u-\sqrt{2u}+1\}$ (see, e.g., \cite{62Suz}). So
$\gexp(S)=4(u^2+1)(u-1)$. Note that $\gexp_3(G)\geqslant 9$ and $\gexp_3(S)=1$.
Since $\gexp_2(S)=\gexp_2(\Aut S)=4$ and $\gexp_2(G)\geqslant 8$, we have
$2\in \pi(K)$.

Let $p=3$. Then $\gexp_3(S)=9$. Since $18\not\in\omega(G)$, neither $K$ nor
$\overline G/S$ contains elements of order $9$. So
$\gexp_3(K)=\gexp_3(\overline G/S)=3$.
The latter equality yields $u=u_0^3$ and $3\cdot
\omega({}^2B_2(u_0))\subseteq\omega(\overline G)$. Furthermore, $k_3(\varepsilon
q)k_6(\varepsilon q)$ divides $u^2+1$, and therefore it is coprime to $u_0^2+1$.
Hence $q^4+q^2+1$ divides $(u^2+1)/(u_0^2+1)=u_0^4-u_0^2+1$, whence
$q<u_0=u^{1/3}$. The above observations about the 2- and 3-exponents of $S$ and $G$ imply that
$\gexp(S)$ divides $\gexp(G)/18$, and so $$2u^3<\gexp(S)\leqslant
\gexp(G)/18=(q^6-1)(q^2+1)/4<q^8/2<u^{8/3}/2,$$ a contradiction.

Let $3\in R_2(\varepsilon q)$. Then $2\gexp_3(S)\not\in\omega(G)$ and
$k_6(\varepsilon q)\gexp_3(S)\in\omega(G)$. It follows that
$\gexp_3(S)\not\in\omega(K)$ and $3k_6(\varepsilon q)\in\omega(\overline G)$.
By Lemma \ref{l:maut}, we have $3k_6(\varepsilon q)<6u^{1/3}$.

Now let $3\in R_1(\varepsilon q)$. Suppose that $\gexp_3(G)\in\omega(K)$ or
$r_3(\varepsilon q)\in\pi(K)$. Then $\pi(K)\subseteq R_1(\varepsilon q)\cup
R_3(\varepsilon q)$.
Furthermore, $\pi(\overline G/S)\cap(\{p\}\cup R_4(q))=\varnothing$ by Lemma
\ref{l:rest_aut}(iv). So $p(q^2+1)/2\in\omega(S)$ and
$(q^3+\varepsilon)/2\in\omega(\overline G)$. If $p(q^2+1)/2$ divides $u-1$, then
$G$ contains an element of order $3p(q^2+1)/2$ or $pr_3(\varepsilon q)r_4(q)$,
which is not the case. It follows that \begin{equation}\label{e:2b2}
p(q^2+1)/2=u+\eta\sqrt{2u}+1\end{equation} for some $\eta$, and in particular
$u\geqslant 32$. Similar reasoning shows that $k_6(\varepsilon q)$ divides
$u-\eta\sqrt{2u}+1$. Suppose that $(q^3+\varepsilon)/2\in\omega(G)$. Then
\begin{equation}\label{e:2b21} (q^3+\varepsilon)/2=u-\eta\sqrt{2u}+1.
\end{equation}

If $q=p$, then $\eta=+$. Subtracting \eqref{e:2b21} from \eqref{e:2b2} yields 
$(q-\varepsilon)/2=2\sqrt{2u}$, and so
$(q^2+1)/2=16u+4\varepsilon\sqrt{2u}+1>u+\sqrt{2u}+1$, a contradiction. If
$q>p$, then $\eta=-$ and
$$1/2>\frac{p(q^2+1)}{q^3-\varepsilon}=\frac{u-\sqrt{2u}+1}{u+\sqrt{2u}+1}
\geqslant 25/41,$$ a contradiction. It follows that
$(q^3+\varepsilon)/2\in\omega(\overline G)\setminus\omega(G)$, and by Lemma
\ref{l:maut}, we have $(q^3+\varepsilon)/2<6u^{1/3}$.
If $\gexp_3(G)\not\in\omega(K)$ and $R_3(\varepsilon q)\cap\pi(K)=\varnothing$,
then  $3k_3(\varepsilon q)\in\omega(\overline G)$, and then 
$3q^2/4<3k_3(\varepsilon q)<6u^{1/3}$.

Thus if $p\neq 3$, then $q^2<8u^{1/3}$, and so $$2u^3<\gexp(G)/18\leqslant
q^9/18<8^{9/2}u^{3/2}/18.$$
This shows that $u=2^3$ or $u=2^5$. But then
$q^2<8\cdot 2^{5/3}<32$, which is a contradiction.
\end{proof}

\begin{lemma}
 $S\neq {}^2G_2(u)$.
\end{lemma}

\begin{proof}
Let $S={}^2G_2(u)$, where $u=3^{2m+1}>3$. Then $\mu(S)=\{9, 6, (u+1)/2,
u-1, uq\pm \sqrt{3u}+1\}$ (see, e.g., \cite{93BrShi}). So the exponent of $S$
is equal to $9(u^{3}+1)(u-1)/4$ and the maximal order of an element is equal to
$u+\sqrt{3u}+1$.
Since $\gexp_2(S)=\gexp_2(\Aut S)$, we have
$\gexp_2(G)/2=2(q-\varepsilon)_2\in\omega(K)$.

Suppose that $(q^3+\varepsilon)/2\in\omega(\overline G)$. Then
$(q^3+\varepsilon)/2\leqslant u+\sqrt{3u}+1$ or $(q^3+\varepsilon)/2<6u^{1/3}$
by Lemma \ref{l:maut}.
In either case, $(q^3+\varepsilon)/2<3u/2$, and hence $q^3\leqslant 3u$. Since $\gexp(S)$
divides $\gexp(G)/4$, it follows that  $$4u^4<9(u^3+1)(u-1)=4\gexp(S)<
\gexp(G)<q^9\leqslant 27u^3,$$ a contradiction. Thus $r_2(\varepsilon q)\in
\pi(K)$. In particular, $R_3(\varepsilon q)\cap \pi(K)=\varnothing$.

Every number in $\omega(G)$ that is a multiple of $\gexp_3(G)$ divides
$(q^3+\varepsilon)/2$ or $(q^3-\varepsilon)/2$, and
so $\omega(G)$ does not contain $r_2(\varepsilon q)\gexp_3(G)\gexp_2(G)/2$.
Hence $\gexp_3(G)\not\in\omega(K)$. On the other hand, $\gexp_3(G)k_3(\eta
q)\in\omega(G)$, where
$q\equiv \eta \pmod 3$, and $\gexp_3(G)>3$, and therefore $3k_3(\eta
q)\in\omega(\overline G)$. It is clear that $3k_3(\eta q)\not\in\omega(S)$,
and so $3k_3(\eta q)\leqslant 6u^{1/3}$ by Lemma \ref{l:maut}, whence $q^2\leqslant
8u^{1/3}$. If $u>27$, then the last inequality yields $q^3\leqslant
8^{3/2}u^{1/2}<3u$. By the computation of the previous paragraph, this is not possible. If
$u=27$, then $q^2\leqslant 24$, which is a contradiction.

\end{proof}

\begin{lemma}
 $S\neq{}^3D_4(u)$.
\end{lemma}

\begin{proof}
Assume the contrary. The spectrum of $^3D_4(u)$ is given in Lemma
\ref{l:spec_3d4}, this lemma implies that $k_6(\varepsilon q)$ divides
$k_{12}(u)=u^4-u^2+1$.

Let $\{r,s,r_6(\varepsilon q)\}$ be a coclique in $GK(G)$ and let $r\in\pi(K)$.
Then $s\in\pi(S)$ and $s\not\in R_{12}(u)$, and so  $S$ has a  non-cyclic
abelian $s$-subgroup (the structure of maximal tori of $S$ is described in
\cite{87DerMich}). Applying Lemma \ref{l:nonc}, we have $rs\in\omega(G)$,
a contradiction.
Thus $(\{p\}\cup R_4(q)\cup R_3(\varepsilon q))\cap \pi(K)=\varnothing$. In
particular, there exist numbers $s\in \pi(S)\cap(\{p\}\cup R_4(q))$ and
$r_3'(\varepsilon q)\in \pi(S)\cap R_3(\varepsilon q)$.

Assume that $(\{p\}\cup R_4(q))\cap\pi(\overline G/S)\neq\varnothing$. Then
$\{s, r_3(\varepsilon q), r_6(\varepsilon q)\}$ is a coclique in $GK(S)$ for
every $r_3(\varepsilon q)$, and so
$R_3(\varepsilon q)\subseteq R_3(\eta u)$ for some $\eta\in\{+1,-1\}$. Hence
$S$ has a Hall $R_3(\varepsilon q)$-subgroup and this subgroup is a product of two
isomorphic cyclic groups.
If $p\neq 3$, we have a contradiction by Lemma \ref{l:rest_aut}(iv). If
$p=3$, then $p\in R_1(u)\cup R_2(u)$ and $pr_3(\varepsilon q)\in\omega(S)$, which is
again a contradiction.

Thus $p(q^2+1)/2\in\omega(S)$. The set $\{r, r_3'(\varepsilon q), r_6(\varepsilon
q)\}$ forms a coclique in $GK(S)$ for every $r\in R_4(q)\cup\{p\}$, so
$\{p\}\cup R_4(q)\subseteq R_3(\eta u)$ for some $\eta\in\{+1,-1\}$. Hence
$p(q^2+1)/2$ divides $u^2+\eta u+1$. Then $4(u^2+\eta
u+1)\in\omega(S)\setminus\omega(G)$, a contradiction.
\end{proof}

\begin{lemma}
 $S\neq G_2(u)$.
\end{lemma}

\begin{proof}
Let $S=G_2(u)$, where $u>2$. The spectrum of $G_2(u)$ is given in Lemma
\ref{l:spec_g2}. In particular, this lemma implies that $$\gexp(S)\geqslant
7(u^6-1)/(3,u^2-1),$$ $k_6(\varepsilon q)$ divides $u^2-\tau u+1$ for some
$\tau\in\{+1,-1\}$, and $u^2-\tau u+1$ divides $(q^3+\varepsilon)/2$.

Assume that $(q^3+\varepsilon)/2\in\omega(\overline G)\setminus\omega(S)$. Then
$(q^3+\varepsilon)/2<6u^{2/3}$, and therefore $q^3\leqslant 12u^{2/3}$. Hence
$u^6\leqslant \gexp(S)\leqslant \gexp(G)<q^9<12^3u^2$, which yields
$u\leqslant 5$. Then $\pi(\overline G/S)\subseteq \{2\}$, and so
$\omega(\overline G)\setminus\omega(S)$ cannot contain the odd number
$(q^3+\varepsilon)/2$.

We claim that the equalities $p(q^2+1)/2=u^2+\tau u+1$ and $R_2(\varepsilon
q)\cap \pi(K)=\varnothing$ are incompatible.  Assume the contrary.
Then $(q^3+\varepsilon)/2\in\omega(\overline G)$ and by the result of the previous
paragraph, it follows that $(q^3+\varepsilon)/2\in\omega(S)$. Hence
$(q^3+\varepsilon)/2=u^2-\tau u+1$.
If $p<q$, then $\tau=-$ and
$$1/2>\frac{p(q^2+1)}{q^3+\varepsilon}=\frac{u^2-u+1}{(u^2+u+1}\geqslant 7/13,$$
because $q\geqslant 7$ and $u\geqslant 3$. If $p=q$, then $\tau=+$. Subtracting
$u^2-u+1$ from $u^2+u+1$ yields $(q-\varepsilon)/2=2u$, and so
$q=4u+\varepsilon$. Then $q^2=(4u+\varepsilon)^2>u^2+u+1>q^2+1$, a contradiction.

Assume that $\pi(K)\cap R_3(\varepsilon q)\neq\varnothing$. Then
$R_2(\varepsilon q)\cap \pi(K)=\varnothing$ and $p(q^2+1)/2\in\omega(S)$.
By the above, $p(q^2+1)/2\neq u^2+\tau u+1$, and hence either $p(q^2+1)/2$ divides
$u^2-1$ (in which case $v=2$), or $p(q^2+1)/2$ divides $v(u-\eta)$ for some $\eta$ and
$v\in R_4(q)$.
In either case, $v\not\in R_3(\varepsilon q)$ and $p$ divides $u^2-1$. The group
$S$ has a Frobenius subgroup with kernel of order $u^2$ and cyclic complement of
order $u^2-1$ \cite[Lemma 1.4]{04CaoChenGr.t},
so $pr_3(\varepsilon q)\in\omega(G)$, a contradiction.

Thus $k_3(\varepsilon q)\in\omega(\overline G)$. Since $k_3(\varepsilon q)>7$,
we apply Lemma \ref{l:rest_aut}(ii) to obtain that there exists
$r'_3(\varepsilon q)\in R_3(\varepsilon q)\cap \pi(S)$.
Assume that $p=3$. Then $v\neq 3$ and $t(3,S)=2$, but $\{3,r_6(\varepsilon q),
r_3'(\varepsilon q)\}$ is a coclique in $GK(S)$, a contradiction.
By Lemma \ref{l:rest_aut}(iii, iv), it follows  that $\{p\}\cup R_4(\varepsilon q)\cup
R_3(\varepsilon q)$ is disjoint from $\pi(\overline G/S)$.

Suppose that $k_3(\varepsilon q)$ does not divide $u^2+\tau u+1$. Then
$k_3(\varepsilon q)$ divides $v(u^2-1)$, and hence every prime $r$ that is
not adjacent to $r_3(\varepsilon q)$ in $GK(G)$
does not divide $v(u^2-1)$ and does not lie in $\pi(K)$. So
$p(q^2+1)/2=u^2+\tau u+1$ and $R_2(\varepsilon q)\cap \pi(K)=\varnothing$, a
contradiction.
Thus $k_3(\varepsilon q)$ divides $u^2+\tau u+1$.

We claim that $k_4(q)\neq v^l$. Otherwise $v\geqslant 5$ and $l=2$ by Lemma
\ref{l:k3v}, and so $(q^2-1)/2=v^2-1$. In particular, $q+\varepsilon$ divides
$v^2-1$, and therefore
$(u^2-\tau u+1, q+\varepsilon)$ divides $(u^2-\tau u+1, v^2-1)$. The last
number does not exceed 3. Since $(u^2-\tau u+1)_3\leqslant 3$ and
$(q^2-\varepsilon q+1)_3=3$ if $(q+\varepsilon,3)=3$, we conclude that $u^2-\tau
u+1$ divides $q^2-\varepsilon q+1$.
So
$$u^2-\tau u+1\leqslant q^2-\varepsilon q+1<2(q^2+1)=4v^2,$$
whence $u=v$. Then $$u^2=(q^2+1)/2<5(q^2-q+1)/8\leqslant 5(u^2+u+1)/8,$$ which
is a contradiction because $u\geqslant 5$.
Thus we can take $r_4'(q)\in R_4(q)\setminus\{v\}$. 

Suppose that $p\in \pi(K)$. Then $p(u^2-1)\in\omega(G)$. Since $u^2-1$ is
divisible by $3$ or $8$ and $(q^2+1)/2$ is not divisible by any of these
numbers, we have that $u^2-1$ divides $q^2-1$.
So $r_4'(q)\in\pi(K)$. If $R_2(\varepsilon q)\cap \pi(K)\neq \varnothing$, then
$pr_2(\varepsilon q)r_4(q)\in\omega(G)\setminus \omega(S)$. Hence
$u^2-\tau u+1=(q^3+\varepsilon)/2$. It follows that $(u^2-1, (q+\varepsilon)/2)$
divides $(u^2-1, u^2-\tau u+1)=(3,u+\tau)$, and so $u^2-1$ divides
$6(q-\varepsilon)$. Then
$$q^3-1\leqslant 2(u^2-\tau u+1)<4(u^2-1)\leqslant 24(q+1),$$ a contradiction.

Thus $p\in \pi(S)$ and $p$ divides $u^2-1$. If $r_4'(q)\in\pi(K)$, then
$r_4'(q)(u^2-1)\in\omega(G)$.
If $r_4'(q)\in\pi(S)$, then $r_4'(q)$ divides $u^2-1$. In either case, one of the
numbers $3pr_4'(q)$ and $4pr_4'(q)$ lies in $\omega(G)$. This contradiction completes the proof.

\end{proof}

\begin{lemma}
 $S\neq F_4(u)$.
\end{lemma}
\begin{proof}
Let $S=F_4(u)$. By Lemma \ref{l:spec_f4}, it follows that
$$\gexp(S)\geqslant 13\cdot\frac{(u^{12}-1)(u^4+1)}{(2,u-1)^2}$$ and
$k_6(\varepsilon q)$ divides $u^4-u^2+1$ or $u^4+1$, and $v=2$ in the latter case.

The group $S$ is unisingular \cite{03GurTie}. Furthermore, every element of $S$
of order $u^4-u^2+1$ lies in a subgroup isomorphic to $^3D_4(q)$, and so it has 
a fixed point in every module in characteristic different from $v$ \cite[Proposition
2]{13Zav.t}. This implies that either $\pi(K)\subseteq R_2(\varepsilon q)$, or
$k_6(\varepsilon q)$ divides $u^4+1$ and $v=2$.
In both cases, $v\not\in \pi(K)$ since otherwise $v$ is adjacent to every number
in $\pi(S)\setminus\{v\}$, and then by Lemmas \ref{l:rest_k}(ii) and
\ref{l:rest_aut}(v),  it is adjacent to every number in $\pi(G)\setminus\{v\}$
in $GK(G)$.

If $(3,q+\varepsilon)=3$, then $k_6(\varepsilon
q)-1=(q-2\varepsilon)(q+\varepsilon)/3$ is not divisible by 4, so
$k_6(\varepsilon q)$ cannot be equal to $u^4-u^2+1$, and if $v=2$, then it
cannot be equal to $u^4+1$ either.
Hence $q^2-q+1\leqslant u^4+1$, and  therefore
\begin{equation}\label{e:f4qu}q^2\leqslant  7u^4/6.\end{equation}

Suppose that $b=(q^3+\varepsilon)/2\in\omega(S)$. Then $b=u^4-u^2+1$, or $v=2$
and $b=u^4+1$. In either case,
$$\gexp(S)/b\geqslant \frac{13(u^{12}-1)}{4}> \frac{13u^{12}}{5}.$$
On the other hand, $\gexp(G)/b\leqslant 6q^6/5$ by Lemma \ref{l:expc3}. In view
of \eqref{e:f4qu}, it follows that $13u^{12}<6q^6\leqslant 7^3u^{12}/6^2$. This is
a contradiction since $13\cdot6^2>7^3$.

Suppose that $b\in\omega(\overline G)$. By the results of the preceding paragraph and
Lemma \ref{l:maut}, we have $b<6u^{4/3}$, and so $q^3\leqslant
12u^{4/3}$. Then $$3u^{16}<\gexp(S)\leqslant \gexp(G)<q^9<12^3u^4,$$ a
contradiction.

Thus $r_2(\varepsilon q)\in\pi(K)$, and hence $R_3(\varepsilon
q)\cap\pi(K)=\varnothing$. Since $k_3(\varepsilon q)>7$, the set
$R_3(\varepsilon q)\cap\pi(S)$ is not empty.
Let $r_3'(\varepsilon q)\in R_3(\varepsilon q)\cap\pi(S)$. By Lemma
\ref{l:rest_k}(iv), the number $r_3'(\varepsilon q)$ cannot divide the order of
a proper parabolic subgroup. It  cannot divide $u^4-u^2+1$ by the above reasoning either.
Hence $r_3'(\varepsilon q)$ divides $u^4+1$ and $k_6(\varepsilon q)$ divides
$u^4-u^2+1$, and so $\pi(K)\subseteq R_2(\varepsilon q)$.

Observe that $p\neq 3$ because $9\in\omega(S)\setminus\mu(S)$. Lemma
\ref{l:rest_aut}(iv) yields $(p\cup R_4(q))\cap\pi(\overline
G/S)=\varnothing$. Hence $a=p(q^2+1)/2$ lies in $\omega(S)$.
If $v\not\in R_4(q)$, then $a$ divides a number that is a multiple of $u^2-1$, which is
impossible since one of the numbers $a$ and $2a$ lie in $\mu(G)$ and neither $3$ nor $2$ divides $a$. If $v\in R_4(q)$,
then $v\geqslant 5$ and $u\equiv 1\pmod 4$, so $a$ divides one of the numbers
$25(u+1)$, $v(u^2+1)(u+1)$ and $v(u^3+1)$. If $p\in R_2(u)\cup R_4(u)$, then one
of the numbers $v(u^2-1)$ and $v(u^2+1)(u-1)$ lies in
$\omega(S)\setminus\omega(G)$. This shows that $p(q^2+1)/2=v(u^3+1)/2$. Then
$$vu^{16}/4<\gexp(S)\leqslant \gexp(G)<a^4<v^4u^{12},$$  whence
$u^4<4v^3$, which is a contradiction since $v\geqslant 5$.
\end{proof}

\begin{lemma} $S\neq E_6^\tau(u)$.
\end{lemma}

\begin{proof}

Assume that $S=E_6^\tau(u)$. Then $k_6(\varepsilon q)$ divides one of the
numbers $u^4+1$, $u^4-u^2-1$, and $k_9(\tau u)$, with $v=2$ in the first case.

Suppose that at least one of the numbers $k_6(\varepsilon q)$ and
$k_3(\varepsilon q)$, say $k$, divides $u^4+1$ or
$u^4-u^2+1$. Then
$2q^2/7\leqslant u^4$, whence $q\leqslant (7/2)^{1/2}u^2$. Lemma \ref{l:exp_e}
yields  $$u^{22}<\gexp(S)\leqslant \gexp(G)<q^9\leqslant
(7/2)^{9/2}u^{18}.$$ Hence $u^4<(7/2)^{9/2}<281$, and so $u\leqslant 4$.
We know that $k\geqslant 19$ and every prime divisor of $k$ is congruent to 1
modulo 3. It follows that either $u=4$ and $k=241$, or $u=3$ and $k=73$. Now
Lemma \ref{l:k3} implies that $u=3$, $k=73$ and $q=9$. This a contradiction with the assumption
$v\neq p$.

Thus we can assume that $k_6(\varepsilon q)$ divides $k_9(\tau u)$.
Suppose that $b=(q^3+\varepsilon)/2$ lies in $\omega(\overline G)$. Then it does not exceed $2u^6$ by Lemma \ref{l:maut}, and so by Lemma \ref{l:expc3}, we have 
$$u^{22}<\gexp(S)\leqslant \gexp(G)<5(2u^6)^3,$$ which yields $u^4<40$. This forces $u=2$, and hence $k_6(\varepsilon q)$ is equal to $73$ if $\tau=+$, and $19$ if $\tau=-$. 
Lemma \ref{l:k3} shows that then $q=9$ or $q=7$
respectively. In either case, we see that  $17\in\omega(S)\setminus\omega(G)$, a contradiction. Thus 
$(q^3+\varepsilon)/2$ does not lie in $\omega(\overline G)$. Then $r_2(\varepsilon q)\in\pi(K)$ and $R_3(\varepsilon q)\cap \pi(K)=\varnothing$.

Let $r\in R_3(\varepsilon q)\cap\pi(S)$. If $v\not\in R_2(\varepsilon q)$, then Lemma \ref{l:rest_k} implies that $r$ cannot divide the order of a proper parabolic subgroup of $S$.
If $v\in R_2(\varepsilon q)$, then $rv\not\in\pi(S)$. In either case, $r$ divides $u^4+1$ or $u^4-u^2+1$. In particular, $7\not\in R_3(\varepsilon q)$ because $7\in \pi(S)$ and $7$ divides $v(u^6-1)$.
By Lemma \ref{l:rest_aut}(ii), we conclude that $k_3(\varepsilon q)$ lies in $\omega(S)$ and divides $u^4+1$ or $u^4-u^2+1$. We argued previously that this is impossible, and so the proof is complete. 
\end{proof}

\begin{lemma} $S\neq E_7(u), E_8(u)$, $^2F_4(u)$.

\end{lemma}

\begin{proof} Assume the contrary. We show that the condition $k_6(\varepsilon q)\in \omega(S)$ implies that $\gexp(G)<\gexp(S)$, at least for sufficiently large $u$.

Let $S=E_8(u)$. Then  $k_6(\varepsilon q)$ divides $k_i(u)$ for some $i\in\{15, 20, 24, 30\}$, and therefore
$q^2/4\leqslant k_6(\varepsilon q)<2u^8$. By Lemma \ref{l:exp_e}, it follows that 
$$\gexp(G)<q^9<8^{9/2}u^{36}<2^{14}u^{36}< 2u^{80}<\gexp(S).$$

Let $S=E_7(u)$. Then $k_6(\varepsilon q)\leqslant k_i(u)$,
where $i\in\{7,9,14,18\}$, and so $q^2/4<2u^6$. Hence
$$\gexp(G)<q^9<8^{9/2}u^{27}<2^{14}u^{27}< 3u^{48}<\gexp(S).$$

Let $S={}^2F_4(u)$, where $u=2^{2m+1}$ and $m\geqslant 1$. Then $k_6(\varepsilon q)$ divides one of the numbers $u^2-u+1$, $u^2\pm\sqrt{2u^3}+u\pm\sqrt{2u}+1$. Assume that $m\geqslant 2$. 
Then $u^2/\sqrt{2u^3}\geqslant 4$, and so
$$u^2+\sqrt{2u^3}+u+\sqrt{2u}+1\leqslant 4^{2m+1}+4^{2m}+\dots+1=(4^{2m+2}-1)/3<4u^2/3.$$  It follows from Lemma \ref{l:k3} that
$16q^2/51<k_6(\varepsilon q)<4u^2/3$, and hence $q<(17/4)^{1/2}u$.  If $\gexp(S)\leqslant \gexp(G)$, then applying  Lemma \ref{l:exp_e} yields
$$16u^{9}(u-1)/3<\gexp(S)\leqslant \gexp(G)< q^9<(17/4)^{9/2}u^9,$$ whence $u\leqslant 3\cdot (17/4)^{9/2}/16<127$, and so $u\leqslant 2^{5}$.

Now let $u=2^3, 2^5$ and find the numbers that can be equal to $k_6(\varepsilon q)$. Recall that $k_6(\varepsilon q)\geqslant 19$ and every prime divisor of this number is congruent to 1 modulo $3$.
These numbers are $19$, $37$, $109$ if $u=2^3$ and $331$, $61$, $13\cdot 61$, $1321$ if $u=2^5$. By Lemma \ref{l:k3}(iv,v), we see that either $u=2^3$ and $q=7$, or $u=2^5$ and $q=31$.
We have $13\in\pi(S)\setminus\pi(G)$ in the first case and $11\in\pi(S)\setminus\pi(G)$ in the second. This contradiction completes the proof.

\end{proof}

Thus the proof of Theorem \ref{t:2} is complete. We now prove Theorem \ref{t:1}. As already mentioned in Introduction, if $p=2$ or $q=3$, then the statement of Theorem \ref{t:1} follows from 
\cite{12Sta.t} and \cite{97ShiTan} respectively.
For all other $q$, Lemma \ref{l:str} implies that a group $G$ isospectral to $L$ has the only nonabelian composition factor $S$. This factor is not an alternating or sporadic group by \cite[Theorems
1 and 2]{09VasGrMaz.t}, and applying Theorem \ref{t:2}, we have that $S$ is a group of Lie type in characteristic $p$. Then \cite[Theorem 2]{15VasGr1} asserts that $S$ is isomorphic to $L$. By
\cite[Theorem 1.1]{15Gr}, the solvable radical of $G$ is trivial, and so $G$ is an almost simple group with socle $L$. Now applying \cite[Theorem 1]{16Gr.t} and \cite[Theorem 1]{18Gr.t}
completes the proof.

%\section*{Acknowledgements}

\end{document}